\documentclass{lms}

\usepackage[all,cmtip]{xy}
\usepackage{amsmath, amssymb}
\usepackage{amsrefs}
\usepackage{stmaryrd}
\usepackage{tikz}
\usetikzlibrary{arrows}
\usetikzlibrary{calc}
\usetikzlibrary{decorations.markings}
\usetikzlibrary{decorations.pathmorphing}
\usetikzlibrary{positioning}

\tikzset{vertex/.style={circle,draw,fill=black,inner sep=0pt, minimum
width=2pt}}

\newtheorem{theorem}{Theorem}[section]
\newtheorem{lemma}[theorem]{Lemma}
\newtheorem{proposition}[theorem]{Proposition}
\newtheorem{corollary}[theorem]{Corollary}

\newnumbered{definition}[theorem]{Definition}
\newnumbered{example}[theorem]{Example}
\newnumbered{xca}[theorem]{Exercise}

\newnumbered{remark}[theorem]{Remark}

\newcommand{\len}{\mathop{\rm length}\nolimits}
\newcommand{\Fix}{\mathop{\rm Fix}\nolimits}
\newcommand{\Stab}{\mathop{\rm Stab}\nolimits}

\newcommand{\diam}{\mathop{\rm diam}\nolimits}
\newcommand{\vol}{\mathop{\rm covol}\nolimits}
\newcommand{\Mod}{\mathop{\rm Mod}\nolimits}
\newcommand{\Out}{\mathop{\rm Out}\nolimits}
\newcommand{\poly}{\mathop{\rm poly}\nolimits}
\newcommand{\core}{\mathcal{C}}
\newcommand{\hatcore}{\widehat{\mathcal{C}}}
\newcommand{\ebox}[1]{\llbracket #1 \rrbracket}
\newcommand{\id}{\mathrm{id}}

\title[Uniform independence for Dehn twists in $\Out(F_r)$]{Uniform independence for Dehn twist automorphisms of a free group}

\author{Edgar A. Bering IV}

\classno{20E36 (primary), 20F65 (secondary)}

\extraline{This article is based upon the author's thesis submitted in partial
fulfillment of the requirements for a Ph. D. at the University of Illinois at
Chicago}

\begin{document}
\maketitle

\begin{abstract}
McCarthy's Theorem for the mapping class group of a closed hyperbolic surface
states that for any two mapping classes $\sigma,\tau \in \Mod(S)$ there is some
power $N$ such that the group $\langle \sigma^N,\tau^N\rangle$ is either free
of rank two or abelian, and gives a geometric criterion for the dichotomy. The
analogous statement is false in linear groups, and unresolved for outer
automorphisms of a free group. Several analogs are known for exponentially
growing outer automorphisms satisfying various technical hypothesis. In this
article we prove an analogous statement when $\sigma$ and $\tau$ are linearly
growing outer automorphisms of $F_r$, and give a geometric criterion for the
dichotomy. Further, Hamidi-Tehrani proved that for Dehn twists in the mapping
class group this independence dichotomy is \emph{uniform}: $N=4$ suffices. In a
similar style, we obtain an $N$ that depends only on the rank of the free
group.
\end{abstract}

\section{Introduction}

In the study of the analogy among linear groups, mapping class groups of
surfaces, and outer automorphisms of free groups, the Tits alternative is a
central achievement. McCarthy~\cite{mccarthy} and Ivanov~\cite{ivanov}
independently established a Tits alternative for mapping class groups.
McCarthy's proof involves a more exact result for two generator subgroups
(quoted below); an analogous statement is false for linear groups, the
Heisenberg group is a counterexample. It is currently unknown whether
$\Out(F_r)$ behaves like a linear group or a mapping class group in this
setting, though there are many partial results, and this article adds another.

McCarthy's theorem for two-generator subgroups of the mapping class group of a
surface $\Sigma$ can be viewed through the lens of a compatibility condition
for geometric invariants associated to a pair of mapping classes. Recall that a
mapping class $\sigma \in \Mod(\Sigma)$ is \emph{rotationless} if every
periodic homotopy class of curve is fixed. Associated to a
rotationless mapping class is a decomposition of $\Sigma$ into invariant
surfaces of negative Euler characteristic $\Sigma_i$ and annuli $A_j$, so that
(up to isotopy) $\sigma|_{\Sigma_i}$ is either identity or pseudo-Anosov, and $\sigma|_{A_j}$
is some power of a Dehn twist about the core curve of $A_j$. The supporting
lamination $\lambda$ of $\sigma$ is the union of the core curves of the
non-trivial Dehn twist components (thought of as measured laminations with
atomic measure equal to the absolute value of the twist power on the core curve) and the attracting measured laminations
of the pseudo-Anosov components.

\begin{theorem}[McCarthy]
	Suppose $\sigma,\tau \in \Mod(\Sigma)$ are mapping classes of a closed
	hyperbolic surface $\Sigma$. Then there is an $N$ such that $\langle
	\sigma^N,\tau^N\rangle$ is either abelian or free of rank two.
	Moreover, $\langle \sigma^N,\tau^N\rangle\cong F_2$ exactly when
	$i(\lambda,\mu) > 0$, where $\lambda$ and $\mu$ are the supporting
	measured laminations of rotationless powers of $\sigma$ and $\tau$
	respectively.
\end{theorem}

Subsequent work of Hamidi-Tehrani~\cite{hamidi-tehrani} showed that when
$\sigma$ and $\tau$ are Dehn twists, $N$ can be chosen independent of $\sigma$,
$\tau$, and the surface. Building on this work and work of Fujiwara~\cite{fujiwara},
Mangahas showed that if $\langle\sigma,\tau\rangle$ is not virtually abelian,
then there is a $p$ depending only on the surface such that one of $\langle
\sigma^p,\tau^p\rangle$, $\langle\sigma^p,\tau\sigma^p\tau^{-1}\rangle$,
$\langle\sigma^p,\tau^p\sigma^p\tau^{-p}\rangle$, or $\langle
\tau^p,\sigma^p\tau^p\sigma^{-p}\rangle$ is free of rank two, which implies that
subgroups of $\Mod(S)$ which are not virtually abelian have uniform exponential
growth, and the exponential growth rate depends only on $S$~\cite{mangahas}.
Parallel results for $\Out(F_r)$ are unknown, and the main theorem of this
article is a step towards them.

Using algebraic laminations an analogous result can be obtained for two
generator subgroups of $\Out(F_r)$ when both generators are exponentially
growing; this was first done by Bestvina, Feighn, and Handel~\cite{bfh-trees}
for pairs of fully irreducible outer automorphisms (with a novel proof using
currents by
Kapovich and Lustig~\cite{kapovich-lustig}), and for exponentially growing
outer automorphisms satisfying certain technical hypotheses by
Taylor~\cite{taylor-out} and Ghosh~\cite{ghosh}. The techniques involved
depend, in one way or another, on the existence of an attracting lamination for
both generators. These approaches therefore do not apply to polynomially
growing outer automorphisms, which have no laminations. Nevertheless, both Clay
and Pettet~\cite{clay-pettet} and Gultepe~\cite{gultepe} prove that the
subgroup of $\Out(F_r)$ generated by powers of ``sufficiently independent''
Dehn twists is free of rank two. Gultepe shows that any two Dehn twists
satisfying a hypothesis on the geometry of their action on a certain complex
generate a free group (without needing to pass to a power); while Clay and Pettet
use tree theoretic methods to work with a larger family of twists, at the
expense of a non-uniform power. (Unlike surface theory, there is no one-to-one
correspondence between trees and laminations~\cite{paulin}.)

In this article we make use of tree theoretic methods (one of which is a
variation on Clay and Pettet's technique, itself an analog of Hamidi-Tehrani's
methods). Two simplicial
trees with $F_r$ action $A$ and $B$ are compatible if there is a tree $T$ with
equivariant surjections $T\to A$ and $T\to B$ that collapse edges.
Guirardel~\cite{guirardel-core} introduced a geometric core of two trees and a
notion of intersection number for these trees which measures compatibility.
Guirardel shows that two simplicial trees are compatible if and only if $i(A,B)
= 0$ for this intersection number. Compatibility is exactly the notion needed to prove an analog of
McCarthy's theorem for linearly growing outer automorphisms of $F_r$. Once
more, certain periodic behavior poses a technical obstacle, but this can be
avoided by passing to a uniform power.

\begin{theorem}[Main Theorem]\label{thm:dt-out-mccarthy}
	Suppose $\sigma$ and $\tau$ are linearly growing outer automorphisms of
	$F_r$.
	For $N=(48r^2-48r+3)|GL(r,\mathbb{Z}/3\mathbb{Z})|$ the subgroup
	$\langle \sigma^N,\tau^N\rangle$ is
	either abelian or free of rank two. Moreover, the latter case holds
	exactly when $i(A,B)>0$ for the Bass-Serre trees $A$ and $B$ of
	efficient representatives of Dehn-twist powers of $\sigma$ and $\tau$.
\end{theorem}

We first introduce the relevant background facts regarding trees and their cores
in Section~\ref{sec:tree-core}; and the necessary parts of the theory of
$\Out(F_r)$ in Section~\ref{sec:out}. The reader familiar with this
theory can safely skim these sections for our notational conventions.
To motivate the development of the tools needed in the proof of the main
theorem, 
we examine a series of guiding examples, including the case of commuting twists and
a setting similar to that considered by Clay and
Pettet~\cite{clay-pettet} in Section~\ref{sec:examples}. The theme
of the proof of the main theorem is to use the core: when it is a tree, it is a small
tree mutually fixed by both automorphisms, and gives a commuting realization of
the automorphisms. Should it fail to be a tree this failure will provide the geometric
information needed to play ping-pong and find powers generating a free group.
Sections~\ref{sec:sgos} and~\ref{sec:tgos} explore the geometric information
obtained in detail, using the core to construct a simultaneous topological
model of both tree actions. Finally, Section~\ref{sec:main-theorem} completes
the proof of the main theorem.

\section{Trees and cores}
\label{sec:tree-core}

A \emph{simplicial tree} is a contractible 1-dimensional cell
complex.  A tree can be given a metric by identifying each 1-cell with an
interval $[a,b]$ (colloquially assigning each 1-cell a length), though the
metric and CW-topologies will not agree in general. In this article we will
always use the metric topology and if not otherwise specified we will use the
metric given by assigning each 1-cell length one (this is often known as the
path metric). A metric tree is uniquely geodesic, for any two points in
$p,q$ the geodesic from $p$ to $q$ is the unique embedded arc joining $p$ and
$q$. We will denote geodesics $[p,q]$ in this article, and use the convention
that these geodesics are oriented; this treats $[q,p]$ as distinct from $[p,q]$
though they are the same set-wise. For an oriented geodesic $e$, $\bar{e}$
denotes its reverse.

\begin{definition} 
	A \emph{simplicial $F_r$-tree} $T$ is an effective right action of the free
	group $F_r$ on a metric tree $T$ by isometries.
\end{definition}

All trees in this article will be $F_r$-trees. We say an $F_r$-tree is
\emph{minimal} if there is no proper invariant subtree $T'\subseteq T$;
\emph{free} when the action is free; \emph{irreducible} when it is
minimal, not a line, and the action does not fix an end; and \emph{small} if the stabilizer of each
edge is trivial or cyclic. Minimal small $F_r$-trees are
irreducible~\cite{culler-morgan}. A metric on an $F_r$-tree gives it a
\emph{covolume}, $\vol(T)$, the sum of the lengths of edges in the quotient
$T/F_r$. Associated to an $F_r$-tree $T$ is a length function
$\ell_T: F_r\to\mathbb{R}_{\geq 0}$ given by
\[ \ell_T(g) = \inf_{x\in T}\{ d(x,x\cdot g)\}. \]

Culler and Morgan~\cite{culler-morgan} give a systematic treatment of (a
generalization) of minimal $F_r$-trees via the associated length functions. For
a fixed group element $g\in F_r$ the set 
\[ C_g^T = \{ x\in T | d(x,x\cdot g) = \ell_T(g)\}\]
is always non-empty and is called \emph{the characteristic set of
$g$}. (When the tree $T$ is clear from context we suppress the superscript.)
Elements with $\ell_T(g) > 0$ are called \emph{hyperbolic} and in this case
$C_g^T$ is a line on which $g$ acts by translation by $\ell_T(g)$. This action
gives $C_g^T$ a natural orientation, and rays contained in $C_g^T$ are referred
to as either positive or negative according to this orientation (n.b.
$C_{g^{-1}}^T$ has the reverse orientation, and gives the opposite
classification to rays). There is a detailed relationship between length
functions and axes elaborated on by Culler and Morgan, we need only a small
piece here.

\begin{lemma}[\cite{culler-morgan}] \label{lem:cmfd}
	Suppose $\ell(gh) \geq \ell(g)+\ell(h)$.
	Then there is a point $p\in C_g^T\cap C_{gh}^T$ such that $[p,p\cdot
	g]\subseteq C_{gh}^T$.
\end{lemma}
\begin{proof} 
	Culler and Morgan give a detailed construction of a fundamental
	domain for $C_{gh}^T$ in all cases. In the cases guaranteed by the
	hypothesis on the length function, this Culler-Morgan fundamental
	domain contains the desired arc.
\end{proof}

Length functions provide a complete isometry invariant for irreducible
$F_r$-trees, and embed the space of $F_r$-trees into $\mathbb{R}^{F_r}$ (one
can restrict to conjugacy classes). The length function of any irreducible tree
is non-zero, so this embedding projectivizes. The space of projective classes
of free simplicial $F_r$-trees is projective Culler-Vogtmann outer space,
$CV_r$; its closure $\overline{CV}_r$ in $\mathbb{PR}^{F_r}$ is
compact~\cite{culler-morgan}. Outer automorphisms act on length functions
by pointwise composition, for $\phi \in \Out(F_r)$ and $\ell : F_r\to
\mathbb{R}$ define $(\phi\ell)(g) = \ell(\phi(g))$, and this gives an action of
$\Out(F_r)$ on $CV_r$ by homeomorphisms that extends to an action on
$\overline{CV}_r$.

\subsection{Very small trees and bounded cancellation}

The work of Cohen and Lustig combined with that of Bestvina and Feighn
characterizes the $F_r$-trees representing projective classes in $\overline{CV}_r$
as the space of all \emph{very small} real
trees~\cites{cohen-lustig,outer-limits}. (Real trees generalize simplicial
trees, but are not needed for this article.)

\begin{definition}
	A $F_r$-tree $T$ is \emph{very small} if it is minimal, small, and has
	\begin{enumerate}
		\item \emph{No obtrusive powers:} for all $g\in F_r\setminus\{\id\}$ and $n$ such that $g^n\neq
			e$, $\Fix(g) = \Fix(g^n)$.
		\item \emph{No tripod stabilizers:} for all $a,b,c\in T$ such that the convex hull
			$H = Hull(a,b,c)$ is not a point or arc, $Stab(H) =
			\{\id\}$.
	\end{enumerate}
\end{definition}

By virtue of their free simplicial approximability, many classical results
about free groups have analogs for very small trees. One indispensable tool is Grayson and
Thurston's bounded cancellation lemma, recorded by Cooper~\cite{cooper}. Fix a
basis for the free group $F_r$ and let $|\cdot|$ denote word length with
respect to this basis. The classical bounded cancellation lemma states

\begin{lemma}[\cite{cooper}]\label{lem:bc1}
	Given an automorphism $f:F_r\to F_r$ there is a constant $C$ such that
	for all $w_1,w_2\in F_r$, if $|w_1w_2| = |w_1|+|w_2|$ then
		\[ |f(w_1w_2)| \geq |f(w_1)|+|f(w_2)|-C. \]
\end{lemma}

Let $T$ be the $F_r$-tree given by the Cayley graph of the fixed basis. An automorphism
$f:F_r\to F_r$ induces a Lipschitz equivariant map $\tilde{f}:T\to T$;
$\tilde{f}$ is the lift of some homotopy equivalence of a wedge of circles
representing $f$. With the unit length metric, $|\cdot|$ gives the arc length
for geodesics based at the identity. Lemma \ref{lem:bc1} implies that the
geodesic from the identity to $w_1w_2$ is
sent to the $\frac{C}{2}$ neighborhood of the geodesic from the identity to
$f(w_1w_2)$.
Since $\tilde{f}$ is equivariant, we conclude that for all finite geodesics
$[p,q]\subseteq T$, $f([p,q])$ is in the $\frac{C}{2}$ neighborhood of the
geodesic $[f(p),f(q)]$. This property generalizes to
equivariant maps between trees.

\begin{definition}\label{def:bc}
	An equivariant continuous map $f: S\to T$ between $F_r$-trees has
	bounded cancellation with constant $C$ if for all geodesics
	$[p,q]\subseteq S$, $f([p,q])$ is in the $C$ neighborhood of the $T$
	geodesic $[f(p),f(q)]$.
\end{definition}

In this form Bestvina, Feighn, and Handel give a bounded cancellation lemma for
very small trees.

\begin{lemma}[\cite{bfh-trees}*{Lemma 3.1}]\label{lem:bc2}
	Suppose $T_0$ is a free simplicial $F_r$-tree and $T$ a very small
	$F_r$-tree, and $f: T_0\to T$ is an equivariant Lipschitz map. Then $f$
	has a bounded cancellation constant $C(f)$ satisfying $C(f) \leq
	Lip(f)\vol(T_0)$.
\end{lemma}

Their proof uses free simplicial approximation to bootstrap this result from
Lemma \ref{lem:bc1}. This lemma in turn implies a form of bounded cancellation
for length functions of very small trees, reminiscent of the form of Lemma
\ref{lem:bc1} (Kapovich and Lustig state a similar lemma, but with subtly
different hypotheses~\cite{kapovich-lustig-bc}).

\begin{lemma}\label{lem:bclf}
	Suppose $T$ is a very small $F_r$-tree and $\Lambda$ a basis for $F_r$.
	There is a constant $C(\Lambda, T)$ such that for all $g,h\in F_r$, if
	$|gh|_\Lambda = |g|_\Lambda+|h|_\Lambda$ and $gh$ is cyclically reduced
	with respect to $\Lambda$, then
		\[ \ell_T(gh) \geq \ell_T(g)+\ell_T(h) - C(\Lambda,T). \]
	Further, $C(\Lambda, T) \leq 6r\inf{Lip(f)}$ where the infemum is taken
	over surjective Lipshitz maps $f:S_\Lambda \to T$ from the universal
	cover of a wedge of $r$ circles marked by the basis $S_\Lambda$.
\end{lemma}

\begin{proof}
	Let $S_\Lambda$ be the universal cover of a wedge of $r$ circles
	with the circles marked by the basis
	$\Lambda$ where all edges have length one. Suppose $f:S_\Lambda\to T$
	is an equivariant Lipschitz surjection. (Such maps always exist: pick
	zero cells $\ast\in S_\Lambda$ and $\star\in T$, define $f:S_\Lambda^{0} \to T$ on
	the zero skeleton by $f(\ast\cdot g) = \star\cdot g$ and extend
	linearly and equivariantly over edges. Since $S_\Lambda$ has finitely
	many edge orbits, this extension is Lipschitz. Moreover, $f$ is
	surjective since $T$ is minimal.) By Lemma~\ref{lem:bc2}, $f$ has
	bounded cancellation. Let $B$ be the bounded cancellation
	constant for $f$. Suppose $g,h\in F_r$ satisfy $|gh|_\Lambda =
	|g|_\Lambda+|h|_\Lambda$ and $gh$ is cyclically reduced. We will show that there is a constant $C$
	depending on $\Lambda$ and $T$
	such that for all $q\in S_\Lambda$,
	\begin{equation}\label{eqn:bcc-lf}
		d(f(q),f(q\cdot gh)) \geq \ell_T(g)+\ell_T(h) - C.
	\end{equation}
	Since $f$ is equivariant and surjective, this implies the conclusion.

	We will establish Equation \ref{eqn:bcc-lf} by showing that for any
	$q \in S_\Lambda$ there is a $p\in C_{gh}^{S_\Lambda}$ so that, for
	auxilliary constants $C'$ and $C''$,
	\begin{equation}
		d(f(q),f(q\cdot gh)) \geq d(f(p),f(p\cdot gh)) - C',
		 \label{eqn:bcc-lf-1}
	\end{equation}
	and for all $p\in C_{gh}^{S_\Lambda}$,
	\begin{equation}
		d(f(p),f(p\cdot gh)) \geq \ell_T(g)+\ell_T(h) - C''.
		\label{eqn:bcc-lf-2}
	\end{equation}

	\emph{Proof of Equation \ref{eqn:bcc-lf-1}.} Let $p$ be the point of
	$C_{gh}^{S_\Lambda}$ closest to $q$. The geodesic $[q,q\cdot gh]$
	contains the points $p$ and $p\cdot gh$. Consider the convex hull
	in $T$ of $f(q),f(p),f(q\cdot gh),$ and $f(p\cdot gh)$
	(Figure \ref{fig:bcc-cxh}).
	\begin{figure}
		\begin{center}
			\begin{tikzpicture}[font=\small]
	\node [vertex] (q) at (0,0) {};
	\node [vertex] (p) at (1,1) {};
	\node [vertex] (pgh) at (3,-1) {};
	\node [vertex] (qgh) at (4,0) {};

	\draw (p) -- node [midway,right] {$\leq B$} (1,0);
	\draw (pgh) -- node [midway,right] {$\leq B$} (3,0);
	\draw (q) -- (qgh);

	\node [left=2pt of q] {$f(q)$};
	\node [above=2pt of p] {$f(p)$};
	\node [below=2pt of pgh] {$f(p\cdot gh)$};
	\node [right=2pt of qgh] {$f(q\cdot gh)$};
\end{tikzpicture}
 			\caption{A convex hull in $T$.}
			\label{fig:bcc-cxh}
		\end{center}
	\end{figure}
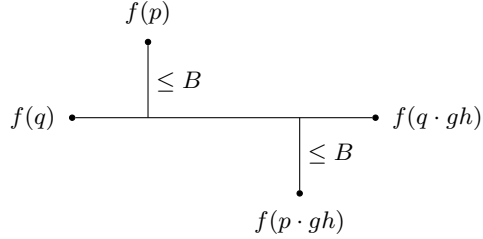
	Since the map $f$ has bounded cancellation, both $f(p)$ and $f(p\cdot
	gh)$ are in the $B$ neighborhood of the geodesic $[f(q),f(q\cdot
	gh)]\subset T$, and we have
		\[ d(f(q),f(q\cdot gh))\geq d(f(p),f(p\cdot gh)) - 2B. \]

	\emph{Proof of Equation \ref{eqn:bcc-lf-2}.} Suppose now that $p\in
	C_{gh}^{S_\Lambda}$. We claim that it suffices to establish the inequality for the
	point $c$ that is the endpoint of the Culler-Morgan fundamental domain
	(Lemma~\ref{lem:cmfd})
	for the action of $gh$ on $C_{gh}^{S_\Lambda}$. To make this claim we
	first need to know the lemma applies. Since $gh$ is reduced and cyclically reduced, the word
	length equals the translation length of $gh$ on $S_\Lambda$, so that
	\[\ell_{S_\Lambda}(gh) = |gh|_\Lambda = |g|_\Lambda+|h|_\Lambda \geq
	\ell_{S_\Lambda}(g)+\ell_{S_\Lambda}(h).\]
	Thus Lemma~\ref{lem:cmfd} applies and there is a $c\in
	C_{gh}^{S_\Lambda}$ such that $[c,c\cdot g]\subseteq
	C_{gh}^{S_\Lambda}$.

	Continuing with the claim, without loss of
	generality we may assume that $c$ is between $p$ and $p\cdot gh$ on
	$C_{gh}^{S_\Lambda}$ by translating by $gh$ as needed. Consider the
	convex hull of $f(p),f(c),f(p\cdot gh),f(c\cdot gh)$ in $T$. Let $x$ be
	the point on the geodesic $[f(p),f(c\cdot gh)]$ closest to $f(c)$ and
	$y$ the point closest to $f(p\cdot gh)$. Since $f$ has bounded
	cancellation, $d(f(c),x),d(f(p\cdot gh),y) \leq B$. We consider two
	cases, when $d(f(p),x) < d(f(p),y)$ and $d(f(p),y)\leq d(f(p),x)$. In
	both cases it will be important to note that, as $f$ is equivariant and
	the action is by isometry,
	$d(f(p),f(c)) = d(f(p\cdot gh),f(c\cdot gh))$.

	In the first case, since $d(f(p),x) + d(x,f(c)) = d(f(p\cdot
	gh),y)+d(y,f(c\cdot gh))$ and $x$ is on the geodesic $[f(c),f(c\cdot
	gh))]$ we have
	\begin{align*}
		d(f(p),f(p\cdot gh)) &= d(f(p),x)+d(x,y)+d(y,f(p\cdot gh)) \\
		&= d(f(c\cdot gh),y)-d(f(c),x)+d(x,y)+2d(y,f(p\cdot gh)) \\
		&\geq  d(f(c\cdot gh),f(c)) - 2B.\\
	\end{align*}

	In the second case, the geodesic $[f(p),f(c)]$ contains $[y,x]$ (which
	may be a point), so we have $d(f(p),y)+d(y,x)+d(x,f(c)) = d(f(p\cdot
	gh), y) + d(y,x) + d(x,f(c\cdot gh))$, and we calculate
	\begin{align*}
		d(f(p),f(p\cdot gh)) &= d(f(p),y) + d(y,f(p\cdot gh)) \\
		&= d(f(c\cdot gh),x)-d(f(c),x)+2d(y,f(p\cdot gh)) \\
		&\geq d(f(c),f(c\cdot gh)) - 2B. \\
	\end{align*}

	Hence, for any $p\in C_{gh}^{S_\lambda}$,
	\[ d(f(p),f(p\cdot gh)) \geq d(f(c),f(c\cdot gh)) - 2B \]
	and it remains to show that this is bounded below by
	translation lengths. 	
	
	By construction, $c\cdot g$ is on the geodesic $[c, c\cdot gh]$.  Consider the image of
	$c, c\cdot g,$ and $c\cdot gh$ in $T$ and the geodesic triangle they
	span. Let $x\in T$ be the midpoint of this triangle
	(Figure \ref{fig:bcc-tri}).
	\begin{figure}
		\begin{center}
			\begin{tikzpicture}[font=\small]
	\node [vertex] (p) at (-1,0) {};
	\node [vertex] (x) at (0,0) {};
	\node [vertex] (pg) at (1,1) {};
	\node [vertex] (pgh) at (1,-1) {};

	\draw (p) -- (x);
	\draw (pgh) -- (x);
	\draw (pg) -- (x);

	\node [left=2pt of p] {$f(c)$};
	\node [above right=2pt of pg] {$f(c\cdot g)$};
	\node [below right=2pt of pgh] {$f(c\cdot gh)$};
	\node [below left=2pt of x] {$x$};
\end{tikzpicture}
 			\caption{The triangle $f(c),f(c\cdot g), f(c\cdot gh)$
			in $T$.}
			\label{fig:bcc-tri}
		\end{center}
	\end{figure}
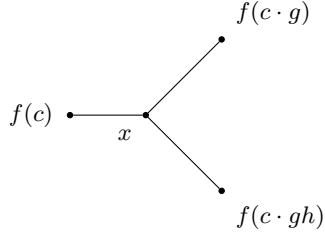
	The bounded cancellation of $f$ implies that $d(x,f(c\cdot g)) \leq B$.
	We have
	\begin{align*}
		d(f(c),f(c\cdot gh)) &= d(f(c),f(c\cdot g))+d(f(c\cdot
		g),f(c\cdot gh)) - 2 d(x,f(c\cdot g)) \\
		&\geq \ell_T(g)+\ell_T(h)-2B \\
	\end{align*}
	establishing Equation \ref{eqn:bcc-lf-2} with $C'' = 4B$.

	Combining Equations \ref{eqn:bcc-lf-1} and \ref{eqn:bcc-lf-2}, we
	have for all $q\in S_\Lambda$
		\[ d(f(q),f(q\cdot gh)) \geq \ell_T(g)+\ell_T(h) -6B \]
	and therefore,
		\[ \ell_T(gh) \geq \ell_T(g)+\ell_T(h) - 6B. \]

	Finally, we note that this proof holds for all equivariant Lipschitz
	surjections $f: S_\Lambda\to T$, and by the previous bounded
	cancellation lemma $B\leq Lip(f)\cdot
	\vol(S_\Lambda) = Lip(f)\cdot r$. Taking an infimum over equivariant Lipschitz
	surjections $f: S_\Lambda\to T$ define $ C(\Lambda,T) =
	6r\inf\{Lip(f)\}$. We conclude
		\[ \ell_T(gh) \geq \ell_T(g)+\ell_T(h) -C(\Lambda,T) \]
	where the constant $C$ depends only on the basis and the very small
	tree $T$.
\end{proof}

To apply this lemma effectively, it is useful to know when a good choice of
basis exists or otherwise obtain control over the Lipshitz maps $f :
S_\lambda\to T$. Lemma~\ref{lem:good-basis} is one example of such control.

\subsection{Bass and Serre's arboretum}
\label{sec:bs-theory}

Bass and Serre~\cite{trees} developed a detailed structure theory for groups
acting on simplicial trees that relates the tree action to a generalization of an
amalgamated product known as a graph of groups. Cohen and Lustig note that this
theory applies equally well to metric trees~\cite{cohen-lustig}. Below we recall key results of the
theory and fix notation.

A \emph{graph} $\Gamma$ is a collection of vertices $V(\Gamma)$, edges
$E(\Gamma)$, initial and terminal vertex maps $o, t : E\to V$, and an
involution $\bar{\cdot} : E\to E$, satisfying $\bar{e}\neq e$ and $o(\bar{e}) =
t(e)$.  When there is a unique edge with $u = o(e)$ and $v = t(e)$ we will
sometimes refer to $e$ as $(u, v)$. An assignment of lengths $d:E\to
\mathbb{R}_{\geq 0}$ satisfying $d(e) = d(\bar{e})$ and $d(e)> 0$ for all $e\in
E$ is a \emph{metric} on $\Gamma$.  These edges are referred to as
\emph{oriented edges}, and a graph $\Gamma$ has a metric space realization by
taking a point for each vertex, and attaching an interval of length $d(e)$
joining $o(e)$ and $t(e)$ for a set of representatives for the orbits of the
involution $\bar{\cdot}$. An \emph{orientation} of a graph $\Gamma$ is a set of
orbit representatives for the involution. When working with graphs if not
otherwise specifying a metric we will use the metric that assigns all edges
length one.

A simplicial tree $T$ can be given a graph structure by taking branch points as
vertices and adding a pair of edges $(p,q) = \overline{(q,p)}$ for each pair
of vertices $p,q\in T^{(0)}$ such that $[p,q]$ is a 1-simplex. Assigning
lengths to one-simplices induces a metric on $T$ and a metric graph
structure. The tree $T$ with this metric is
the metric realization of this metric graph structure. When it is important to do so we will distinguish between a
simplicial tree and a graph structure arising from a simplicial tree by calling
the latter a \emph{graphical tree}.  A group $G$ acting on $T$ by simplicial
isomorphism (isometry)
naturally acts on this (metric) graph structure, and we say this action is \emph{without
inversion} if for all $e\in E(T)$ and $g\in G$, $e\cdot g \neq \bar{e}$. An
action with inversion can be turned into an action without inversion by
subdividing $T$.

\begin{definition}
	A \emph{(metric) graph of groups} is a pair $(G,\Gamma)$ where $\Gamma$
	is a connected (metric) graph, and $G$ is an assignment of groups to
	the vertices and edges of $\Gamma$ satisfying $G_{e} = G_{\bar{e}}$,
	and injections $\iota_e : G_e \to G_{t(e)}$. We will often suppress
	the assignment $G$ and write $\Gamma_e, \Gamma_v,$ etc.
\end{definition}

The following applies to metric graphs of groups~\cite{cohen-lustig} equally
well, but we make only light use of metric trees and can make do without
belaboring the point.

The fundamental theorem of Bass-Serre theory gives an equivalence between
actions on graphical trees and graphs of groups. Given a group $G$ acting on a
graphical tree $T$, the quotient graph $\bar{T}$ has a graph of groups
structure as follows. Pick a maximal subtree $S\subseteq \bar{T}$ and an
orientation $Y$ of $\Gamma$. Define a section $j:\bar{T} \to T$ by first fixing
a lift of $S$, and then for each $e\in Y\setminus E(S)$, define $j(e)$ so
that $o(j(e)) = j(o(e))$; also choose elements $\gamma_e \in G$ so that $t(je)
= \gamma_e j(t(e))$ for these edges. The assignment of $\gamma_e$ is extended
to all of $E(\bar{T})$ by $\gamma_{\bar{e}} = \gamma_e^{-1}$ and $\gamma_e = 1$
for $e\in E(S)$. Let $\chi$ be the indicator function for $E(\bar{T})\setminus
Y$. The graph of groups structure on $\bar{T}$ is given by $G_v = \Stab(j(v))$,
$G_e = \Stab(j(e))$ and the inclusion maps by $\iota_e(a) =
\gamma_e^{\chi(e)-1}a\gamma_e^{1-\chi(e)}$. Different choices of lift and
maximal tree give isomorphic graphs of groups structures on the quotient, we say
two graphs of groups are \emph{equivalent} if they are different quotient
labellings of the same tree.

Starting from a graph of groups $\Gamma$ there is an inverse operation,
which recovers the group $G$ as the \emph{fundamental group} of the graph of
groups, and the tree $T$ that $G$ acts on so that the quotient is
$\Gamma$. This is the \emph{Bass-Serre tree} of $\Gamma$, the
construction depends on a choice of maximal tree, and is unique up to
equivariant isomorphism (isometry in the metric case). We will denote the quotient
graph of groups by $\bar{T}$ and its tree $T$. When working with properties
that are not conjugacy invariant the fundamental domain used will be
specified.

The construction of the fundamental group of a graph of groups sits naturally
in the context of the \emph{fundamental groupoid} of a graph of groups,
introduced by Higgins~\cite{higgins}.

\begin{definition}
	The \emph{fundamental groupoid} $\pi_1(\Gamma)$ of a graph of groups $\Gamma$ is
	the groupoid with vertex set $V(\Gamma)$, generated by the path
	groupoid of $\Gamma$ and the groups $G_v$ subject to the following
	conditions. We require that for each $v\in V(\Gamma)$ the group $G_v$
	is a sub-groupoid
	based at $v$ and that the group and groupoid structures agree.
	Further for all $e\in E(\Gamma)$ and $g\in G_e$, we have
		\[ \bar{e}\iota_{\bar{e}}(g)e = \iota_e(g). \]
	In particular this implies $\bar{e}$ and $e$ are inverse in
	$\pi_1(\Gamma)$.
\end{definition}

By taking the vertex subgroup of $\pi_1(\Gamma)$ at a vertex $v$, we get
the fundamental group $\pi_1(\Gamma,v)$. Changing basepoint results in
an isomorphic group. The group $\pi_1(\Gamma,v)$ can also be described in
terms of maximal trees. Fix a maximal tree $T$, and take the quotient of
$\pi_1(\Gamma)$ by first identifying all vertices and then collapsing all
edges of $T$. As explained by Higgins, it follows from standard results in
groupoid theory that the result is isomorphic to
$\pi_1(\Gamma,v)$~\cite{higgins}.

Let $e = (e_1,e_2,\ldots,e_n)$ be a possibly empty edge path starting at $v$
and $g = (g_0,g_1,\ldots,g_n)$ a sequence of elements $g_i \in G_{t(e_i)}$ with
$g_0 \in G_v$. These data represent an arrow of $\pi_1(\Gamma)$ from $v$
to $t(e_n)$ by the groupoid product
	\[ g_0e_1g_1\cdots e_ng_n. \]
A non-identity element of $\pi_1(\Gamma)$ expressed this way is
\emph{reduced} if either $n=0$ and $g_0\neq \id$, or $n>0$ and for all $i$ such
that $e_i = \bar{e}_{i+1}$, $g_i \notin G_{e_i}^{e_i}$. By fixing appropriate left
transversals, a normal form for arrows of $\pi_1(\Gamma)$ is obtained. For
each edge $e\in E(\Gamma)$, fix a left transversal $S_e$ of the image of $G_e$
in $G_{o(e)}$ containing the identity; by inductively applying the defining
relations a reduced arrow is equivalent to a reduced arrow of the form
	\[ s_0e_0s_1\cdots e_nh \]
with each $s_i\in S_{e_i}$ and $h\in G_{t(e_n)}$. This representation is
unique~\cite{higgins}. By specializing to $\pi_1(\Gamma,v)$ we obtain the
\emph{Bass-Serre normal form} for elements of the fundamental group based at
$v$, with $h\in G_v$. This normal form depends on the choice of
left-transversal, but the edges used do not.

For a conjugacy class $[g] \in \pi_1(\Gamma,v)$, a representative $g$ is
cyclically reduced if it is reduced, $s_0 = \id$, and $g$ has no sub-arrow $g'$
based at $v$ such that $g=cg'c^{-1}$ for $c\in\pi_1(\Gamma,v)$. In
particular, if $o(e_0) = t(e_0) = v$, we have that if $\bar{e}_n = e_0$, then
$h\notin \iota_{e_n}(G_{e_n})$.

When $\pi_1(\Gamma,v)$ is free all vertex and edge
groups are also free. A more refined normal form can be obtained
by fixing an ordered basis $\Lambda$ for $\pi_1(\Gamma,v)$. Using the
lexicographic order induced by $\Lambda$ and the
Nielsen-Schreier theorem we obtain a unique minimal basis for each $G_v$. The
induced order on the minimal bases of the $G_v$ specifies a unique minimal
left Schreier transversal for the image of each $G_e$ with $t(e) = v$. 
Further, using the minimal right Schreier transversal $R_e$ of
$G_e$ in $G_{t(e)}$ with respect to its preferred basis, we obtain a unique
expression of the form
	\[ x_0r_0e_1x_1r_1\cdots e_nx_nr_n \]
where $x_0 \in \iota_{e_n}(G_{e_n})$, each $x_i\in \iota_{e_i}(G_{e_i})$, and $r_i \in
R_{e_i}$, and $x_ir_i$ reduced words with respect to the induced bases of the
vertex groups. 
We call this \emph{the transverse Bass-Serre normal form with respect to $\Lambda$}.

\begin{definition}\label{def:gog-min}
	A graph of groups $\Gamma$ is \emph{minimal} if for every
	connected proper subgraph $\Gamma'$ and $v\in V(\Gamma')$ the
	induced map $\pi_1(\Gamma',v)\to\pi_1(\Gamma,v)$ is not
	surjective.
\end{definition}

\begin{remark}
	This implies that if $v\in V(\Gamma)$
	has valence one in a minimal graph of groups $\Gamma$, then $\iota_e(G_e)$
	is not surjective, for the unique edge $e$ satisfying
	$v=t(e)$. As long as $\pi_1(\Gamma,v) \ncong \mathbb{Z}$ or
	$D_\infty$, the resulting tree $T$ is then an irreducible
	$\pi_1(\Gamma,v)$-tree.
\end{remark}

\begin{proposition}[\cite{cohen-lustig}*{Proposition 9.2}]
	A graph of groups $\Gamma$ is minimal if and only if
	its Bass-Serre tree $T$ is a minimal $\pi_1(\Gamma,v)$ tree.
\end{proposition}

\begin{proof}
	Cohen and Lustig leave this proof to the reader. We include it here.
	Suppose $\Gamma'\subseteq \Gamma$ is a connected proper
	subgraph and $\pi_1(\Gamma',v)\to \pi_1(\Gamma,v)$ is
	surjective. Take a lift of $T'$ (the tree of $\Gamma'$) to $T$. This
	is a $\pi_1(\Gamma',v)$ invariant subtree by construction, and the
	action of $\pi_1(\Gamma,v)$ is induced by inclusion, so
	$T_{\Gamma'}$ is a $\pi_1(\Gamma,v)$ invariant subtree, since
	the inclusion is surjective. Conversely, if $T'\subseteq T$
	is proper and $\pi_1(\Gamma,v)$ invariant, then
	$T'/\pi_1(\Gamma,v)$ is a connected proper subgraph with
	graph of groups fundamental group $\pi_1(\Gamma,v)$, the induced
	inclusion map is an isomorphism.
\end{proof}

To ensure that two minimal graphs of groups with equivariantly isometric
Bass-Serre trees are isomorphic as graphs of groups a certain pathology must be
excluded.

\begin{definition}\label{def:gog-visible}
	Let $\Gamma$ be a graph of groups. A valence two vertex $v\in
	V(\Gamma)$ with $v = t(e_1) = t(e_2)$ is \emph{invisible} if
	$\iota_{e_1}$ and $\iota_{e_2}$ are isomorphisms. If $\Gamma$ has no
	invisible vertices it is a \emph{visible} graph of groups.
\end{definition}

Invisible vertices are readily created by barycentric subdivision of edges and
result in non-isomorphic simplicial structures on the Bass-Serre tree without
changing the equivariant isometry class.

\subsection{Topological models}

Several authors give, in varying stages of development, an approach to building
a topological model of a graph of
groups~\cites{scott-wall,tretkoff,chipman,althoen}. The treatment given by
Scott and Wall is the popular reference~\cite{scott-wall}, though Tretkoff's
account includes a significantly more extensive discussion of the topological
basis of normal forms~\cite{tretkoff}. The definitions given by the various
authors are equivalent in the cellular category, though the language is quite
variable. This section will most closely follow Tretkoff's account.

\begin{definition}
	A \emph{graph of spaces} $\mathcal{X}$ over a graph $\Gamma$ is a collection of cell
	complexes $\mathcal{X}$ indexed by the vertices and edges of $\Gamma$,
	such that $\mathcal{X}_e^m = \mathcal{X}_{\bar{e}}^m$, and cellular
	inclusions $\iota_e: \mathcal{X}_e^m\to \mathcal{X}_{t(e)}$. The
	\emph{total space} of
	$\mathcal{X}$, denoted $X$ is the quotient of the disjoint union 
		\[ \sqcup_{v\in V(\Gamma)} \mathcal{X}_v \sqcup_{e\in
		E(\Gamma)} \mathcal{X}_e\times [0,1] \]
	by the identifications
	\begin{gather*}
		\mathcal{X}_e^m\times [0,1] \to \mathcal{X}_{\bar{e}}^m\times [0,1]
		\qquad (x,t) \mapsto (x,1-t) \\
		\mathcal{X}_e^m\times 1 \to \mathcal{X}_v \qquad (x,1) \mapsto
		\iota_e(x)
	\end{gather*}
\end{definition}

The total space $X$ of a graph of spaces over $\Gamma$ comes with a map
$q:X\to\Gamma$ to the topological realization of $\Gamma$ by $q(\mathcal{X}_v)
= v$ and $q(\mathcal{X}_e^m\times\{t\}) = e(t)$, the point of $e$ at coordinate
$t$ realizing $e$ as the one-cell $[0,1]$. If $X$ is a cell complex with
cellular map $q:X\to \Gamma$ such that the preimages of vertices and midpoints
of edges gives a graph of spaces structure with $X$ as the total space, we say
$q$ induces a graph of spaces structure on $X$. Note that the image of
$\mathcal{X}_e^m\times [0,1]$ in $X$ is the double mapping cylinder on the two
inclusion maps, we denote this image $\mathcal{X}_e$. (Indeed, some authors only require the maps be $\pi_1$
injective and construct the total space with the double mapping cylinder.)
The spaces $\mathcal{X}_e^m$ naturally include into the total space $X$ via the
map $\mathcal{X}_e^m \to \mathcal{X}_e^m\times \{\frac{1}{2}\}$, hence the
superscript $m$ for midpoint.

By taking fundamental groups of the vertex and edge spaces of a graph of spaces
we obtain an associated graph of groups assignment $G$ on $\Gamma$, and with
$x\in \mathcal{X}_v$, $\pi_1(X,x) \cong
\pi_1(\Gamma,v)$. This operation of course has an inverse, given a graph
of groups $\Gamma$ a natural graph of spaces over $\Gamma$ can be constructed from
$K(\Gamma_v,1)$ and $K(\Gamma_e,1)$ spaces. The group of deck transformations of the universal cover
$\tilde{X}$ gives a definition of the fundamental group of
$\Gamma$ that does not require a choice of basepoint or maximal tree.

Tretkoff gives a topological normal form for the homotopy class of a path
relative to the endpoints in a graph of spaces, taking advantage of a
classification of edges in the one skeleton. For a graph of spaces structure
$\mathcal{X}$ with total space $X$, an edge in $X^{(1)}$ is
\emph{$\mathcal{X}$-nodal} if it lies in a vertex space, and
\emph{$\mathcal{X}$-crossing} otherwise. Tretkoff's form makes use of a fixed topological realization of the left
transversals to ensure uniqueness, we need only the topological taxonomy of
edges in the path, as formulated by Bestvina, Feighn, and Handel~\cite{bfh-ii}.

\begin{lemma}[\citelist{\cite{tretkoff}\cite{bfh-ii}*{Section 2.7}}]
	\label{lem:gos-nf}
	Every path in a graph of spaces $X$ is homotopic relative to the
	endpoints to a path of the form (called \emph{normal form})
		\[ v_0H_1v_1H_2\cdots H_nv_n \]
	where each $v_i$ is a (possibly trivial) tight edge path of
	$\mathcal{X}$-nodal edges, each $H_i$ is $\mathcal{X}$-crossing,
	and for all $1\leq i \leq n-1,$ $H_iv_iH_{i+1}$ is not homotopic
	relative to the endpoints to an $\mathcal{X}$-nodal edge path. Any
	two representatives of the homotopy class of a path in normal form have
	the same $n$. A similar statement holds for free homotopy classes of
	loops.
\end{lemma}

The proof of this lemma also illustrates that an edge path can be taken to
normal form by iteratively \emph{erasing a pair of crossing edges}; if
$H_iv_iH_{i+1}$ is homotopic relative to the endpoints to a nodal edge path
$v_i'$ then the subpath $v_{i-1}H_iv_iH_{i+1}v_{i+1}$ is homotopic relative to
endpoints to $v_{i-1}v_i'v_{i+1}$ which can subsequently be tightened. Note
that a path is in normal form if and only if every sub-path is. This should be
compared to the normal form for arrows in the fundamental groupoid of a graph
of groups, indeed one proof of the groupoid normal form is to prove this normal
form and then apply the natural map from the fundamental groupoid of the total
space $X$ to the fundamental groupoid of the graph of groups in question.
\subsection{A core sampler}

Guirardel introduced the core of two real trees with group action to unify and
generalize several intersection and compatibility phenomena in group theory.

\begin{definition}[\cite{guirardel-core}]
	\label{def:core}
	The \emph{core} of two simplicial $F_r$-trees $A$ and $B$, $\core(A,B)$
	is the minimal non-empty closed subset of $A\times B$ with convex
	fibers invariant under the diagonal action of $G$. The \emph{augmented
	core} $\hatcore(A,B) \supseteq \core(A,B)$ is the minimal closed
	connected subset of $A\times B$ with convex fibers invariant under the
	diagonal action.
\end{definition}

\begin{remark}
	If $A$ and $B$ have minimal subtrees $A'$ and $B'$ then the core must
	be contained in $A'\times B'$.
\end{remark}

Guirardel works in the much more general setting of group actions on real
trees, but in this article we do not need to leave the cellular category;
Guirardel shows if $A$ and $B$ are simplicial $G$-trees then $\core(A,B)$ is a
square subcomplex of $A\times B$~\cite{guirardel-core}*{Proposition 2.6}.
Further, for irreducible trees, the core is always non-empty, though it is not
always connected.

The diagonal action of $F_r$ on $\core(A,B)$ induces a notion of covolume,
while this notion is not well behaved in general, in the simplicial setting $\vol(\core)$ is the total
metric area of $\core/F_r$ (the number of squares when all edges of $A$ and $B$
have length one). Without a condition on the edge stabilizers of $A$
and $B$ this may be infinite, but we are concerned with the other extreme.

\begin{definition}\label{def:isect}
	The \emph{intersection number} of two simplicial $F_r$-trees $A$ and $B$ is
	\[ i(A,B) = \vol(\core(A,B)).\]
\end{definition}

For simplicial $F_r$-trees, the intersection number quantifies the (non)-existence of a
common refining tree. Given two simplicial $F_r$ trees $A$ and $B$, we say that
$T$ is a \emph{common refinement} of $A$ and $B$ if there are equivariant
surjections $f_A: T\to A$ and $f_B: T\to B$ that \emph{preserve alignment}, the image of
every geodesic $[p,q]$ is $[f_S(p),f_S(q)]$ with $S$ either $A$ or $B$. These
maps arise from equivariantly collapsing edges. 

\begin{theorem}[\cite{guirardel-core}*{Theorem 6.1}]\label{thm:core-refine}
	Simplicial $F_r$-trees $A$ and $B$ have a common refinement if and only
	if $i(A,B) = 0$. In this case $\hatcore(A,B)$ is a common refinement.
\end{theorem}

In a previous paper~\cite{compat-trees} we give some equivalent
characterizations of compatibility for irreducible $F_r$-trees that are useful for explicit computations
(one of these generalizes a criterion of Behrstock, Bestvina, and
Clay~\cite{bbc}).
Let $e\subset T$ be an oriented edge in a simplicial $F_r$-tree. Let
$\delta_e^+$ be the connected component of $T\setminus e^\circ$ containing
$t(e)$. The \emph{asymptotic horizon} of $e$ is the set of group elements
\[ \ebox{e} = \{ g\in F_r | C_g^T\cap \delta_e^+ \text{ is a positive ray}\} \]

\begin{lemma}[\cite{compat-trees}]
	\label{lem:compat-tree}
	Suppose $A$ and $B$ are irreducible simplicial $F_r$-trees. The following are
	equivalent.
	\begin{enumerate}
		\item $A$ and $B$ are not compatible.
		\item There are edges $a\in E(A)$ and $b\in E(B)$ such that the
			four sets
			\[ \ebox{a}\cap\ebox{b}, \ebox{\bar{a}}\cap\ebox{b},
			\ebox{a}\cap\ebox{\bar{b}},\ebox{\bar{a}}\cap\ebox{\bar{b}}
			\]
			are all non-empty.
		\item There are group elements $g, h\in F_r$ such that 
			\begin{align*} 
				\ell_A(gh) = \ell_A(gh^{-1}) >
				\ell_A(g)+\ell_A(h)&\qquad\text{and}\qquad
				\ell_B(gh)\neq \ell_B(gh^{-1}) \\
				&\text{or}\\
				\ell_B(gh) = \ell_B(gh^{-1}) >
				\ell_B(g)+\ell_B(h)&\qquad\text{and}\qquad
				\ell_A(gh)\neq \ell_A(gh^{-1}). \\
			\end{align*}
	\end{enumerate}
\end{lemma}

The third condition is called \emph{incompatible combinatorics} because of its
implications about the combinatorial arrangement of axes and $A$ and $B$.	

\subsection{The Bass-Serre case}

While not all useful stabilizer restrictions are retained by the core of
compatible trees, when $A$ and $B$ are compatible Bass-Serre trees for graph of
groups decompositions of $G$ the structure theory of the core permits a
very explicit description of the augmented core.

\begin{lemma}\label{lem:core-bs}
	Suppose $\bar{A}$ and $\bar{B}$ are minimal visible graphs of
	groups with fundamental group $G\ncong \mathbb{Z}$ or
	$\mathbb{Z}/2\mathbb{Z}\ast\mathbb{Z}/2\mathbb{Z}$,
	and compatible Bass-Serre trees $A$ and $B$.
	The augmented core $\hatcore(A,B)$ is then then
	the Bass-Serre tree for a graph of groups $\Gamma$ with fundamental
	group $G$, and the
	edge groups of $\Gamma$ are in the set of conjugacy classes of the
	edge groups of $\bar{A}$ and $\bar{B}$. Moreover, $\bar{A}$
	and $\bar{B}$ are equivalent
	to graphs of groups $\bar{A}'$ and $\bar{B}'$ so that
	
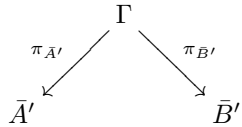
\[\xymatrix{
			& \Gamma\ar[ld]_{\pi_{\bar{A}'}}\ar[rd]^{\pi_{\bar{B}'}} & \\
			\bar{A}' & & \bar{B}'
	}\]
	where $\pi_{\bar{A}'}$ and $\pi_{\bar{B}'}$ are quotient maps
	that collapse edges.
\end{lemma}

\begin{proof}
	Guirardel proves that in this case the core is a common refinement and
	so $\hatcore(A,B)$ is a simplicial $F_r$-tree (Theorem
	\ref{thm:core-refine}). Moreover, by the convexity of the fibers of the
	projection maps, the edges of $\hatcore$ are of three forms
		\[ \{ v_A\} \times e_B, e_A\times
		\{v_B\}, \mbox{ or } \Delta\subseteq
		e_A\times e_B. \]
	where $v_T$ and $e_T$ are vertices and edges in the trees $A$ and
	$B$. Further, using the equivariant projections from the core $\pi_A$
	and $\pi_B$, we calculate stabilizers for each edge,
	$e\in \hatcore(A,B)$
		\[ \Stab_{\hatcore(A,B)}(e) =
		\Stab_{A}(\pi_{A}(e))\cap
		\Stab_{B}(\pi_{B}(e)). \]
	Suppose $\pi_{A}(e) = a \in E(A)$. We claim
		\[ \Stab_{\hatcore(A,B)}(e) =
		\Stab_{A}(a). \]
	Indeed, suppose there is some $g\in \Stab_{A}(a)$ but not
	in $\Stab_{B}(\pi_B(e))$. Let $p\in a$ be the midpoint and let
	$q\in \pi_{B}(e)$ be any point. The point $(p,q)$ is in the
	interior of $e$, and since $g$ is not in the stabilizer, $(p\cdot g,
	q\cdot g) = (p,q\cdot g)$ is disjoint from $e$. Both $(p,q), (p,g\cdot
	q)\in \pi^{-1}_{A}(p)$, which is convex. However, the
	path in $\hatcore(A,B)$ must pass through
	$o(e)$ or $t(e)$, neither of which is in $\pi^{-1}_{A}(p)$,
	a contradiction. Symmetrically, if $\pi_{B}(e) = b \in
	E(B)$ we find
		\[ \Stab_{\hatcore(A,B)}(e) =
		\Stab_{B}(b). \]
	
	The remainder of the lemma is then immediate from standard facts in
	Bass-Serre theory, with $\hatcore(A,B)$ the
	Bass-Serre tree of the desired graph of groups $\Gamma$. The graphs of
	groups $\bar{A}'$ and $\bar{B}'$ come from choosing a maximal tree and
	lift in $\Gamma$ and $\hatcore(A,B)$, and projecting.
\end{proof}

\begin{remark}
	This characterizes the edge groups of compatible graphs of groups: An
	edge group $\bar{A}_e$ is either conjugate to some $\bar{B}_e$ or contained within
	a conjugate of some $\bar{B}_v$, and vise-versa.
\end{remark}
\section{Outer automorphisms}
\label{sec:out}

By definition, the outer automorphism group $\Out(F_r) =
\mathrm{Aut}(F_r)/\mathrm{Inn}(F_r)$ of a free group $F_r$
is the automorphism group modulo the inner automorphisms. We briefly review
various topological perspectives on elements of $\Out(F_r)$, the classification
by growth, and some details about representatives of outer automorphisms of
linear growth.

\subsection{Topological representatives and growth}

Let $\Gamma$ be the realization of a graph with $\pi_1(\Gamma,v) = F_r$. An immersed
path $\gamma: [0,1]\to \Gamma$ is \emph{tight} if any lift
$\tilde{\gamma}:[0,1]\to\tilde{\Gamma}$ is an embedding. Since $\tilde{\Gamma}$
is a tree, it is immediate that every immersed path is homotopic relative to
the endpoints to a unique tight path, called its \emph{tightening}. Given a
path $\gamma$ we denote the tightening $[\gamma]$. Similarly, a closed loop is
tight if it is tight for every choice of basepoint, and is freely homotopic to
a unique tightening (a fundamental domain for the action of $\gamma_\ast\in
\pi_1(\Gamma)$ on the universal cover $\tilde{\Gamma}$, with basepoint chosen
on the axis of $\gamma_\ast$), the tightening of a loop $\gamma$ is denoted
$[[\gamma]]$. Two paths $\gamma$ and $\delta$ are composable if the end of
$\gamma$ equals the start of $\delta$, and their composition is denoted
$\gamma\delta$; if $\gamma$ is a based loop $\gamma^{-1}$ denotes its reverse
and $\gamma^m$ its $m$-fold concatenation for $m\in \mathbb{Z}$ (when $m=0$
this is a constant path at the basepoint of $\gamma$). A loop $\gamma$ is primitive if there is no
$\gamma'$ such that $[\gamma] = [\gamma'^m]$ for some $m>1$.  We will assume
from here on that all paths have endpoints at the vertices of $\Gamma$.

Given an outer automorphism $\sigma\in\Out(F_r)$, we can realize $\sigma$ as a
homotopy equivalence $\hat{\sigma}:\Gamma\to\Gamma$. Such a realization is
referred to as a \emph{topological representative}; particularly nice
topological representatives are indispensable in the analysis of outer
automorphisms. 

The growth of an outer automorphism is measured in terms of a topological
representative. We say $\sigma$ is \emph{exponentially growing} if there is
some loop $\gamma\subseteq \Gamma$ such that
$\ell_{\Gamma}([[\hat{\sigma}^n(\gamma)]])$ is bounded below by an exponential
function, and that $\sigma$ is \emph{polynomially growing} if there is some $d$
such that $\ell_{\Gamma}([[\hat{\sigma}^n(\gamma)]]) \in O(n^d)$ for all
loops $\gamma\subseteq \Gamma$. This classification does not depend on the
choice of topological representative, as demonstrated by Bestvina, Feighn, and
Handel~\cite{bfh-i}; the choice does matter for the details of the exponent in
the exponentially growing case, however we are not concerned with exponentially
growing outer automorphisms in this article.

Polynomially growing outer automorphisms can exhibit a certain amount of
finite-order periodic behavior which results in significant technical
headaches. These phenomena can be removed by passing to a uniform power. A
polynomially growing outer automorphism $\sigma$ is \emph{unipotent} if the
induced action on the first homology $H_1(F_r, \mathbb{Z})$ is a unipotent
matrix. Bestvina, Feighn, and Handel proved that any polynomially growing outer
automorphism that acts trivially on $H_1(F_r,\mathbb{Z}/3\mathbb{Z})$ is
unipotent~\cite{bfh-ii}*{Proposition 3.5}, so all polynomially growing outer
automorphisms have a unipotent power.

\subsection{Upper triangular representatives and the Kolchin theorem}

Unipotent polynomially growing outer automorphisms have particularly nice
topological representatives. A homotopy equivalence $\hat{\sigma}:\Gamma\to
\Gamma$ is \emph{filtered} if there is a filtration $\emptyset = \Gamma_0\subsetneq
\Gamma_1\subsetneq \cdots \subsetneq \Gamma_k = \Gamma$ preserved by
$\hat{\sigma}$.

\begin{definition}
	A filtered homotopy equivalence $\hat{\sigma}$ is \emph{upper triangular} if
	\begin{enumerate}
		\item $\hat{\sigma}$ fixes the vertices of $\Gamma$,
		\item Each stratum of the filtration
			$\Gamma_i\setminus\Gamma_{i-1} = E_i$ is a single
			topological edge,
		\item Each edge $E_i$ has a preferred orientation and with this
			orientation there is a tight closed path
			$u_i\subseteq\Gamma_{i-1}$  based
			at $t(E_i)$ so that $\hat{\sigma}(E_i) = E_iu_i$.
	\end{enumerate}
\end{definition}

The path $u_i$ is called the suffix associated to $u_i$, and when working with
an upper triangular homotopy equivalences we will always refer to edges of the
filtered graph with the preferred orientation.  Just as paths have tightenings,
if $\hat{\sigma}$ is a filtered homotopy equivalence that satisfies the above
definition except that some $u_i$ is not tight, $\hat{\sigma}$ is
homotopic to an upper triangular homotpy equivalence, also called its
tightening. A filtration assigns to each
edge a \emph{height}, the integer $i$ such that $E\in
\Gamma_i\setminus\Gamma_{i-1}$, and by taking a maximum this definition extends to
tight edge paths. An upper-triangular homotopy equivalence preserves the height of
each edge path.

Every upper triangular homotopy equivalence of a fixed filtered graph evidently
induces a unipotent polynomially growing outer automorphism, and using relative train tracks
Bestvina, Feighn, and Handel show the converse, every unipotent polynomially
growing outer automorphism has an upper triangular
representative~\cite{bfh-i}*{Theorem 5.1.8}. Moreover, for a given filtered
graph $\Gamma$ the upper-triangular homotopy equivalences taken up to homotopy
relative to the vertices form a group under composition. The suffixes for the
inverse are defined inductively up the filtration by $\hat{\sigma}^{-1}(E_i) =
E_iv_i$ where $v_i = \overline{\hat{\sigma}^{-1}(u_i)}$.

A nontrivial path $\gamma\subseteq \Gamma$ is a \emph{periodic Nielsen path}
for $\hat{\sigma}$ if for some $m > 0$, we have $[\hat{\sigma}^m(\gamma)] =
[\gamma]$. If $m = 1$ we call $\gamma$ a \emph{Nielsen path}. An \emph{exceptional
path} in $\Gamma$ is a path of the form $E_i\gamma^m\bar{E}_j$, where $\gamma$ is
a primitive Nielsen path, and $\hat{\sigma}(E_i) = E_i\gamma^p$ and
$\hat{\sigma}(E_j) = E_j\gamma^q$ for $p,q > 0$ and any $m$. For a unipotent
polynomially growing automorphism, every closed periodic Nielsen path is Nielsen~\cite{bfh-ii}*{Proposition
3.16}. If $p\neq q$ we say the exceptional path is \emph{linearly growing}, otherwise
it is an \emph{exceptional Nielsen path}.

Every path $\gamma \subseteq \Gamma$ has a \emph{canonical decomposition} with
respect to an upper triangular $\hat{\sigma}$ into single edges and maximal
exceptional paths~\cite{bfh-ii}*{Lemma 4.26}.

For all of the terms in the previous two paragraphs, when we are dealing with
more than one upper-triangular homotopy equivalence we will specify which
homotopy equivalence is involved, e.g. ``a path $\gamma$ is
$\hat{\sigma}$-Nielsen'' or ``consider the $\hat{\tau}$-canonical decomposition of
$\gamma = \gamma_1\gamma_2\cdots\gamma_k$''.

The analogy between unipotent polynomially growing outer automorphisms and unipotent matrices
stretches beyond having an upper-triangular basis. The classical Kolchin
theorem for linear groups~\cite{kolchin} states that if a subgroup $H\leq GL(n,\mathbb{C})$
consists of unipotent matrices then there is a basis so that with respect to
this basis every element of $H$ is upper triangular with 1's on the diagonal.
There is an analogous theorem for unipotent polynomially growing outer
automorphisms, due to Bestvina, Feighn, and Handel.

\begin{theorem}[\cite{bfh-ii}*{Main Theorem}]\label{thm:kolchin}
	Suppose $H\leq \Out(F_n)$ is a finitely generated subgroup with every
	element unipotent polynomially growing. Then there is a filtered graph $\Gamma$ and a fixed
	preferred orientation such that every $\sigma\in H$ is upper triangular
	with respect to $\Gamma$.
\end{theorem}

\begin{remark}
	Bestvina, Feighn, and Handel use a different definition of
	upper-triangular, allowing that $\sigma(E_i) = v_iE_iu_i$, however our
	definition can be obtained by subdividing each edge and doubling the
	length of the filtration. 
\end{remark}

\subsection{Dehn twists and linear growth}
\label{sec:dtback}

Let $\Sigma$ be a closed hyperbolic surface. Given $\gamma\subseteq\Sigma$ an
essential simple closed curve, consider a homeomorphism $\tau_\gamma
:\Sigma\to\Sigma$ that is the identity outside an annular neighborhood of
$\gamma$ and performs a twist of $2\pi$ on the annulus. Such a homeomorphism is
known as a \emph{Dehn twist}. The induced map $\tau_{\gamma\ast} :
\pi_1(\Sigma)\to\pi_1(\Sigma)$ can be expressed in terms of the graph of groups
decomposition of $\pi_1(\Sigma)$ induced by $\gamma$, and this expression
motivates the following definition for general graphs of groups.

\begin{definition}
	Suppose $\Gamma$ is a graph of groups. Given a fixed 
	collection of edges $\{e_i\} \subseteq E(\Gamma)$ closed under
	the edge involution and $z_{e_i}\in
	Z(G_{e_i})$ satisfying $z_{\bar{e}_i} = z_{e_i}^{-1}$, the \emph{Dehn twist
	about $\{e_i\}$ by $\{z_i\}$}, $D_{z} \in \Out(\pi_1(\Gamma,v))$, is the
	outer automorphism induced by $\tilde{D}_z$ on the fundamental groupoid of
	$\Gamma$, given by
	\begin{align*}
		\tilde{D}_z(e_i) &= e_iz_i^{e_i} \\
		\tilde{D}_z(g) &= g, &g\in G_v, v\in V(\Gamma) \\
		\tilde{D}_z(e) &= e, &e\notin\{e_i\}
	\end{align*}
	The induced outer automorphism does not depend on the choice of
	basepoint.
\end{definition}

Note that $D_z^n = D_{z^n}$, defining $z^n = \{ z_{e_i}^n\}$ for any $n$, and that
any two twists on a fixed graph of groups $\Gamma$ commute. The requirement
that each $z_{e_i} \in Z(G_{e_i})$ is necessary to ensure that the defining
relations of the fundamental groupoid are respected. In turn, when
$\pi_1(\Gamma,v)$ is free a Dehn twist can only twist around edges with cyclic
stabilizers.

\begin{example}
	Let $\Gamma$ be the graph of groups associated to the amalgamated
	product $A\ast_C B$ and $z\in Z(C)$. The twist of $\Gamma$ about its
	edge by $z$ can be represented by $D_z(a) = z^{-1}az$, $a\in A$,
	$D_z(b) = b, b\in B$. Since $A\cup B$ generates $\pi_1(\Gamma,v)$
	this fully specifies the automorphism.

	Let $\mathcal{H}$ be the graph of groups associated to the HNN extension
	$A \ast_C$ and pick $z\in Z(C)$. The twist of $\mathcal{H}$ about its
	one edge by $z$ is represented by $D_z(a) = a$ and $D_z(t) = tz$ with
	$a\in A$ and $t$ the edge of the extension.

	Specializing these examples to splittings of $\pi_1(\Sigma)$ given by
	an essential closed curve in a closed hyperbolic surface $\gamma
	\subseteq \Sigma$, this gives the previously mentioned algebraic
	representation of $\tau_{\gamma\ast}$ as the Dehn twist about the edge
	of the splitting corresponding to $\gamma$ by
	$\gamma_\ast\in\pi_1(\Sigma)$.
\end{example}

\begin{example}[Nielsen automorphisms of $F_r$]
	Consider the graph of groups $\Gamma$ in Figure \ref{fig:gog-nielsen}.
	\begin{figure}
		\begin{center}
			\begin{tikzpicture}[thick,font=\small,every text node
	part/.style={align=center}]
	\node [vertex,label=below:{$v$ \\ $G_{v} = \langle
	a_1,\ldots,a_n\rangle$}] (P) at (0,-1) {};
	\draw (0,0) circle (1cm)  [postaction={decorate},decoration={
    markings,
	mark=at position 0.25 with {\arrow{>}}}];
	\coordinate (T) at (0,1);
	\node [above =2pt of T] {$G_t = \langle z\rangle$ \\ $t$};
\end{tikzpicture}
 			\caption{The graph of groups used to represent Nielsen
			automorphisms.}
			\label{fig:gog-nielsen}
		\end{center}
	\end{figure}
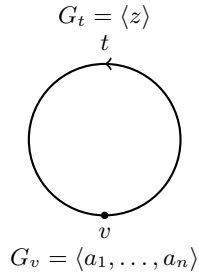
	The edge morphisms for the single edge are given by $\iota_t(z) = a_j$
	and $\iota_{\bar{t}}(z) = a_k$.
	The map $F:\langle x_1,\ldots, x_n\rangle \to \pi_1(\Gamma,v)$ given by
	$F(x_i) = a_i$, $i\neq j$, and $F(x_j) = t$ gives a realization of the
	Nielsen automorphism $\phi(x_i) = x_i, \phi(x_j) = x_kx_j$
	as the Dehn twist about the single edge by $z$.
\end{example}

A Dehn twist outer automorphism has many graph of groups representatives, most
of which are not well suited to analysis using the Guirardel core, due to lots
of extra information. Certain ill-behaved stabilizers, non-minimal graphs,
invisible vertices, and unused edges all cause trouble. Cohen and Lustig
identified a particularly useful class of representatives, called efficient
twists. 

\begin{definition}\label{def:eff}
	A Dehn twist $D$ on a graph of groups $\Gamma$ is \emph{efficient} if
	\begin{enumerate}
		\item $\Gamma$ is minimal, small, and visible,
		\item $D$ twists about every edge (every $z_e\neq \id$),
		\item (\emph{no positively bonded edges}) there is no pair of
			edges $e_1,e_2 \in E(\Gamma)$ such that $v=t(e_1)
			= t(e_2)$, and integers $m,n\neq 0$ with $mn > 0$, such
			that $z_{e_1}^m$ is conjugate in $G_v$ to $z_{e_2}^n$.
	\end{enumerate}
\end{definition}

Cohen and Lustig remark that it is a consequence of these three properties that
$\Gamma$ is necessarily very small. Returning our attention to
$\Out(F_r)$ a Dehn twist outer automorphism $D\in\Out(F_r)$ is one that can be
represented as a Dehn twist of some graph of groups decomposition of $F_r$
(such a decomposition necessarily only twists about those edges with cyclic
edge groups). These outer automorphisms have linear growth (and all outer
automorphisms with linear growth are roots of Dehn
twists~\cite{margalit-schleimer}).

By assigning each edge of a graph of groups $\Gamma$ a positive length, the Bass-Serre
tree $T$ of $\Gamma$ becomes a metric $F_r$-tree. Given a very small graph of
groups $\Gamma$ with fundamental group $F_r$, the
collection of projective classes of all choices of metric on $T$ determines an open
simplex $\Delta(\Gamma)\subseteq\overline{CV}_r$ in projectivized outer space.
If $\Gamma$ is visible and minimal, this simplex is of dimension
$|E(\Gamma)|-1$. When $D$ is an efficient Dehn twist on $\Gamma$, the simplex
$\Delta(\Gamma)$ is completely determined by the dynamics of the action of $D$
on $CV_r$, as shown by Cohen and Lustig~\cite{cohen-lustig}.

\begin{theorem}[\cite{cohen-lustig}*{Theorem 13.2}]\label{thm:parabolic-orbits}
	Suppose $D$ is a Dehn twist in $\Out(F_r)$ with an efficient
	representative on a graph of groups $\Gamma$. Then for all $[T]\in
	CV_r$,
	\[ \lim_{n\to\infty} D^n([T]) = \lim_{n\to\infty} D^{-n}([T]) \in
	\Delta(\Gamma). \]
\end{theorem}

\begin{corollary}\label{cor:et-uniq}
	If $D\in\Out(F_r)$ has an efficient Dehn twist representative,
	then the simplicial structure of the Bass-Serre tree of the
	representative is unique.
\end{corollary}

\begin{proof}
	Suppose $D$ has efficient representatives $D_1$ on $\Gamma_1$ and $D_2$
	on $\Gamma_2$. By the theorem, $\Delta(\Gamma_1) = \Delta(\Gamma_2)$
	since two open simplices which share a point are equal.
	This completes the claim.
\end{proof}

The efficient graph of groups representative of a Dehn twist can be constructed from an
upper-triangular representation. Bestvina, Feighn, and Handel give this
construction in the metric category, using a particular upper-triangular
representation that permits them to compute metric information about the
limit in $\overline{CV}_r$, but the uniqueness of the algebraic structure
permits the calculation from any upper-triangular representation. First note
that an upper-triangular homotopy equivalence grows linearly if and only if
each suffix is Nielsen, and that each edge is either fixed or grows linearly.

To construct the efficient representative from an upper-triangular
representative we need the notion of folding in a tree or graph, due to
Stallings~\cite{stallings-folds}. In a simplicial $F_r$-tree $T$, a \emph{fold}
of two edges $u,v\in T$ with $o(u) = o(v)$ for a linear homeomorphism $\phi:
u\to v$ is the quotient of $T$ by the smallest equivalence relation satisfying
$x\sim \phi(x)$ for all points $x\in u$ and if $x\sim y$ and $g\in F_r$ then
$x.g\sim y.g$. The quotient map of this equivalence $\tilde{f}: T\to T/\sim$ is
called the \emph{folding map}, and the resulting space $T/\sim$ is a $F_r$-tree
(it may be necessary to subdivide to ensure that the action is without
inversions). When the action on the folded tree $T/\sim$ is without inversions,
we get a graph of groups morphism on the quotient $f:\bar{T}\to
\overline{T/\sim}$.

Let $q: T\to \bar{T}$ be a graph of groups quotient map. There is a particular
type of fold we treat in detail. Suppose there is an element $g\in G$ such that
the folding homeomorphism $\phi: u\to v$ is induced by the $g$ action. In this
case $g\in \bar{T}_{o(u)}$ and $g$ conjugates $\Stab(u)$ to $\Stab(v)$. The folded
graph of groups $\overline{T/\sim}$ has the same combinatorial structure as
$\bar{T}$, however $\Stab(u/\sim) = \langle \Stab(u), g\rangle$, so that
$u/\sim$ has a larger edge group. This is
referred to as ``pulling an element in a vertex group over an edge".

By subdividing an edge we may perform a \emph{partial fold} of the
first half of $u$ over $v$. (Partial folding can be discussed in much greater
generality; we require only the midpoint version.) We will often specify a fold
by a pair of edges $u$ and $v$ with $o(u)=o(v)$ in the quotient graph of groups, it is understood that we mean the
equivariant fold of all pairs of lifts $\tilde{u},\tilde{v}$ with $o(\tilde{u})
= o(\tilde{v})$. The
definition of folding generalizes to allow $v$ to be an edge path, and we use
this more general definition.

\begin{lemma}\label{lem:efficient-tt}
	Suppose $\hat{\sigma}:\Gamma\to\Gamma$ is a linearly growing
	upper-triangular homotopy equivalence of a filtered graph $\Gamma$.
	Then there is an $F_r$-tree $T$ and a composition of folds and
	collapses $f : \tilde{\Gamma}\to T$ which
	realizes the outer automorphism represented by
	$\hat{\sigma}$ as an efficient Dehn twist on the graph of groups
	quotient $\bar{T}$.
\end{lemma}

\begin{proof}
	The strategy of the proof is to collapse every fixed edge; in the
	resulting graph of groups, the suffix of the lowest linear edge is in a
	vertex group, and so the suffix can be folded over that edge. Working up the
	filtration in this fashion the result is a graph of groups with cyclic
	edge stabilizers, and by twisting on every edge by the twister
	specified by its suffix; the result is a Dehn twist on this graph which
	represents $\hat{\sigma}$.
	
	The problem with this construction, as just described, is that the
	result may not be efficient: there may be obtrusive powers, and there
	may be positively bonded edges. The first problem is solved by using
	the primitive root of the suffix, but the second requires some work. 
	One could use Cohen and Lustig's algorithm to remove positive bonding,
	however we give a different construction similar to that of Bestvina,
	Feighn, and Handel~\cite{bfh-ii} useful when considering more than one
	Dehn twist. In this construction we first fold certain edges with
	related suffix so that when we carry out the sketch above no positive
	bonding results.

	We assume without loss of generality $\Gamma$ is minimal (that is, the
	quotient of a minimal tree under the $F_r$ action).

	\emph{Step 1: Fold Conjugates}. We construct a series of folds by
	working up the filtration from lowest
	edge to highest. Start with $\Gamma^0 = \Gamma$.
	Suppose the suffix $u_i$ of $E_i$ is of the form
	$\gamma_i[\eta_j^k]\bar{\gamma}_i$ with $k\neq 0$, where $u_j = [\eta_j^{k'}]$ so that
	$\eta_j$ is the primitive Nielsen path associated to $u_j$, $j < i$ and
	$\gamma_i$ a closed path of height at most $i-1$. Since $u_i$ is
	Nielsen and tight we must have $[\hat{\sigma}_{i-1}(\gamma_i)] =
	\gamma_i\eta_j^m$ for some $m\in \mathbb{Z}$ (possibly zero).
	In this case fold the terminal half of $E_i$ over $\bar{\gamma}_i$. Let
	$f_i:\Gamma^{i-1}\to \Gamma^i$ be the folding map in this step. We
	claim the induced homotopy equivalence satisfying $\hat{\sigma}_if_i =
	f_i\hat{\sigma}_{i-1}$ has an upper triangular tightening.  Let $E_i'$ denote
	the unfolded initial half of $E_i$, and filter $\Gamma^i$ by the
	filtration of $\Gamma^{i-1}$ where the $i$th stratum is now $E_i'$. It
	suffices to check that $\hat{\sigma}_i(E_i') = E_i'u_i'$. Indeed, using
	the equation
		\[ f_i\hat{\sigma}_{i-1}(E_i\gamma_i) = \hat{\sigma}_i(E_i')\]
	we have for some $m\in\mathbb{Z}$
		\[ f_i\hat{\sigma}_{i-1}(E_i\gamma_i) =
		f_i(E_i\gamma_i\eta_j^k\bar{\gamma}_i\gamma_i\eta_j^m) =
		E_i'\bar{\gamma}_i\gamma_i\eta_j^k\bar{\gamma}_i\gamma_i\eta_j^m \]
	and so the tightening of $\hat{\sigma}_i$ gives $E_i'$ suffix
	$[\eta_j^{k+m}]$ (for edges other than $E_i'$ the suffix is the same as
	that of $\hat{\sigma}_{i-1}$, which already has upper triangular
	tightening). If the
	suffix $u_i$ of $E_i$ is not of the above form, take $\Gamma^i =
	\Gamma^{i-1}$ and $f_i = \id$.

	Denote the total folding map $f_k\cdots f_0 = f': \Gamma\to \Gamma'$,
	and the tightening of the induced
	automorphism $\hat{\sigma}'$. By construction $\hat{\sigma}'$ is upper
	triangular and has the property that for every two edges $E_i$ and
	$E_j$ with common terminal vertex, if their suffixes have conjugate
	roots then they are of the form $u_i = [\eta^{k_i}]$, $u_j =
	[\eta^{k_j}]$ for positive powers of a primitive Nielsen path $\eta$.

	\emph{Step 2: Fold Linear Families.} Starting now with $\hat{\sigma}'$,
	we perform another sequence of folds to ensure that twisters will not be
	positively bonded. For a primitive Nielsen path $\eta$, the linear
	family associated to $\eta$ is all edges of $\Gamma'$ with suffix
	$[\eta^k]$ for some $k \neq 0$. We now work down the filtration of
	$\Gamma'$. Set $\Gamma_k' = \Gamma'$. If $E_i'$ is in the linear family
	associated to some primitive Nielsen path $\eta$, let
	$E_j$ be the next edge lower than $E_i'$ in the linear family, and fold
	half of $E_i'$ over all of $E_j'$. Denote the fold
	$f_i':\Gamma_i'\to\Gamma_{i-1}'$ in this case; otherwise set
	$\Gamma_{i-1}' =\Gamma_i$ and $f_i' = \id$. Let $\Gamma'' = \Gamma_0'$ be the total result
	of this folding, with total folding map $f_0'\cdots f_k' = f'':\Gamma'\to\Gamma''$, and denote
	the unfolded halves of edges by $E_i''$. (If an edge is not folded we
	will also use $E_i''$ for the edge as an edge of $\Gamma''$). The graph
	$\Gamma''$ is naturally filtered, with the filtration induced by $f''$.
	We claim that the induced homotopy equivalence $\hat{\sigma}'' =
	f''\hat{\sigma}'{f''}^{-1}$ is again upper triangular. Indeed, as in
	the previous case we can calculate the suffixes. For $E_i'$ denote by
	$E_{i_1}',\ldots E_{i_l}'$ the edges in the linear family of $E_i'$ below
	$E_i'$ in descending order, so that $f''(E_i') = E_i''E_{i_1}''\cdots
	E_{i_l}''$. Working inductively up the linear family, a calculation
	similar to the previous step finds
	$\hat{\sigma}''(E_i'') = E_i'' E_{i_1}''\cdots
	E_{i_l}''[f''(\eta)^{k_i-k_{i_1}}]\bar{E}_{i_l}''\cdots\bar{E}_{i_1}'$, and the
	associated primitive Nielsen path to $E_i''$ is $\eta_i'' =
	E_{i_1}''\cdots E_{i_l}''[f''(\eta)]\bar{E}_{i_l}''\cdots\bar{E}_{i_1}''$.

	\emph{Step 3: Collapse and Fold Edge Stabilizers.} From
	$\hat{\sigma}''$ and $\Gamma''$ we can now construct a graph of groups;
	the previous two steps will ensure that no twisters in the result are
	positively bonded. We work up the filtration once more. Let
	$\bar{T}^0$ be the graph of groups constructed from $\Gamma''$ by
	collapsing all edges with trivial suffix. Obtain $\bar{T}^i$ from
	$\bar{T}^{i-1}$ as follows. If $\hat{\sigma}''(E_i'') = E_i''$, set
	$\bar{T}^i = \bar{T}^{i-1}$. If $\hat{\sigma}''(E_i'') =
	E_i''[{\eta_i''}^{k_i''}]$ then obtain $\bar{T}^i$ from $\bar{T}^i$
	by pulling $\eta_i''$ over $E_i''$. By construction $\eta_i''$
	represents an element in a vertex group at some lift $t(E_i'')$. The result is
	$\bar{T}$. The composition of folding maps $f'':
	\Gamma''\to\bar{T}$ induces a Dehn twist $\tilde{\sigma}$ on
	$\bar{T}$ where the system of twisters is given by $z_{E_i''} =
	{\eta_i''}^{k_i''}$. By construction, this twist represents
	$\hat{\sigma}''$ and so $\hat{\sigma}$; moreover the edge stabilizers are not
	conjugate in the vertex groups, as a result of the first two steps;
	therefore the resulting twist is efficient except
	for the possibility of invisible vertices. Invisible vertices are an
	artifact of the graph of groups; removing them gives the desired efficient
	twist.
\end{proof}

\begin{remark}
	It is possible that $\hat{\sigma}$ is upper triangular with respect to
	several different filtrations of $\Gamma$. By fixing a filtration a
	choice is being made, but the choices made do not matter because of Corollary~\ref{cor:et-uniq}.
\end{remark}

\begin{example}
	To illustrate the procedure in Lemma \ref{lem:efficient-tt} we
	calculate the efficient representative of $\sigma\in\Out(F_4)$ given by
	\begin{align*}
		a&\mapsto adbcb^{-1}d^{-1} \\
		b&\mapsto bc \\
		c&\mapsto c \\
		d&\mapsto d.\\
	\end{align*}
	We will start with the upper triangular representative
	$\hat{\sigma}:\Gamma\to\Gamma$ on the rose on 4 petals with topological
	edges named $a, b, c, d$ filtered by reverse alphabetical order and the
	images of edges under $\hat{\sigma}$ given as above. This
	representative has a single linear family $\{a, b\}$ with associated
	primitive Nielsen path $c$.

	\emph{Step 1: Fold Conjugates.} Working up the filtration we find that
	the only edge that needs folding is $a$, we fold half of $a$ over
	$\bar{b}\bar{d}$. This gives the folding map $f': \Gamma\to \Gamma'$
	where $\Gamma'$ is a rose on four petals with edges $(a', b, c, d)$,
	$f'(a) = a' \bar{b}\bar{d}$, and $f'(e) = e$ for $e\neq a$. The induced
	upper triangular representative $\hat{\sigma}' : \Gamma'\to\Gamma'$ is
	given by
	\begin{align*}
		a'&\mapsto a'c^2 \\
		b&\mapsto bc \\
		c&\mapsto c \\
		d&\mapsto d.\\
	\end{align*}
	Indeed, we can verify that $\hat{\sigma}'(a') = a'c^2$ by calculating:
		\[ \hat{\sigma}'(a') = \hat{\sigma}'(a'\bar{b}\bar{d}db)
		= \hat{\sigma}'(f(adb)) = [f(\hat{\sigma}(adb))] =
		[a'\bar{b}\bar{d}dbc^2] = a'c^2. \]

	\emph{Step 2: Fold Linear Families.} Working down the filtration, the
	only edge that requires folding is $a'$: we fold the terminal half over
	$b$. This defines $f'': \Gamma'\to\Gamma''$ where $\Gamma''$ is the
	four petals with edges $(a'',b,c,d)$, $f''(a') = a''b$, and $f''(e) =
	e$ for $e\neq a'$. Calculating $\hat{\sigma}''(a'')$ in a similar
	fashion:
	\[ \hat{\sigma}''(a'') = [f''(\hat{\sigma}'(a'b))] =
	[a''bc^2\bar{c}\bar{b}] = a''bc\bar{b}. \]
	The action of $\hat{\sigma}''$ on the remaining edges is the same as
	that of $\hat{\sigma}'$.

	\emph{Step 3: Collapse and Fold Edge Stabilizers.} Once more working up
	the filtration we first collapse the edges $c$ and $d$ with trivial
	suffix, which gives the graph of groups $\bar{T}^0$ which has two free
	edges $a''$ and $b$, and a vertex with stabilizer $\langle c, d\rangle$.
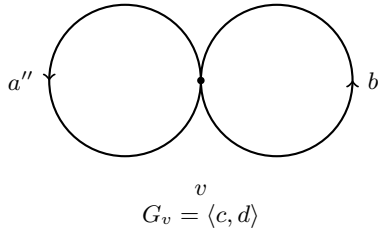
	\begin{center}
		\begin{tikzpicture}[thick,font=\small,every text node
	part/.style={align=center}]
	\node [vertex] (V) at (0,0) {};
	\draw (-1,0) circle (1cm)  [postaction={decorate},decoration={
    markings,
	mark=at position 0.5 with {\arrow{>}}}];
	\draw (1,0) circle (1cm)  [postaction={decorate},decoration={
    markings,
	mark=at position 0.0 with {\arrow{>}}}];
	\coordinate (A) at (-2,0);
	\coordinate (B) at (2,0);
	\node [left =2pt of A] {		$a''$};
	\node [right =2pt of B] {		$b$};
	\node [below =1.2cm of V] {$v$ \\ $G_v = \langle c, d 		\rangle$};
\end{tikzpicture}
 	\end{center}
	Next, we pull $c$ over $b$, and (using the orientation
	$\{\bar{a},\bar{b}\}$ and keeping in mind that in this article we are
	using right actions) obtain $\bar{T}^1$.
	
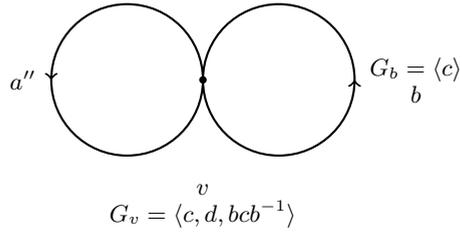
\begin{center}
		\begin{tikzpicture}[thick,font=\small,every text node
	part/.style={align=center}]
	\node [vertex] (V) at (0,0) {};
	\draw (-1,0) circle (1cm)  [postaction={decorate},decoration={
    markings,
	mark=at position 0.5 with {\arrow{>}}}];
	\draw (1,0) circle (1cm)  [postaction={decorate},decoration={
    markings,
	mark=at position 0.0 with {\arrow{>}}}];
	\coordinate (A) at (-2,0);
	\coordinate (B) at (2,0);
	\node [left =2pt of A] {		$a''$};
	\node [right =2pt of B] {$G_{b} = \langle c\rangle$ \\ 
		$b$};
	\node [below =1.2cm of V] {$v$ \\ $G_v = \langle c, d , bcb^{-1}
		\rangle$};
\end{tikzpicture}
 	\end{center}
	Finally, we pull $bc\bar{b}$ over $a''$ and using the same orientation
	for labels arrive at $\bar{T}$.
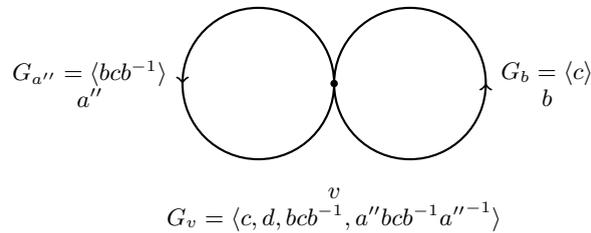
	\begin{center}
		\begin{tikzpicture}[thick,font=\small,every text node
	part/.style={align=center}]
	\node [vertex] (V) at (0,0) {};
	\draw (-1,0) circle (1cm)  [postaction={decorate},decoration={
    markings,
	mark=at position 0.5 with {\arrow{>}}}];
	\draw (1,0) circle (1cm)  [postaction={decorate},decoration={
    markings,
	mark=at position 0.0 with {\arrow{>}}}];
	\coordinate (A) at (-2,0);
	\coordinate (B) at (2,0);
	\node [left =2pt of A] {$G_{a''} = \langle bcb^{-1}\rangle$ \\ 
		$a''$};
	\node [right =2pt of B] {$G_{b} = \langle c\rangle$ \\ 
		$b$};
	\node [below =1.2cm of V] {$v$ \\ $G_v = \langle c, d , bcb^{-1},
	a''bcb^{-1}{a''}^{-1} \rangle$};
\end{tikzpicture}
 	\end{center}
	The Dehn twist representative $\tilde{\sigma}$ is given by the system
	of twisters $z_{a''} = bcb^{-1}$ $z_b = c$. Observe $a'' = ad$.
\end{example}

The upper triangular representative constructed in the previous lemma provides
us with a basis of $F_r$ with small bounded cancellation constant for the
length function on the Bass-Serre tree $T$.

\begin{lemma}\label{lem:good-basis}
	Suppose $\sigma$ is an efficient Dehn twist on the very small graph of
	groups $\bar{T}$, and let $T$ be the Bass-Serre tree. Then there is a
	basis $\Lambda$ for $F_r$ such that the bounded cancellation constant
	constant $C(\Lambda, T)$ from Lemma \ref{lem:bclf} satisfies
	\[ C(\Lambda, T) \leq 6r(2r-2). \]
\end{lemma}

\begin{proof}
	From Lemma \ref{lem:bclf}, we know that $C(\Lambda, T) \leq
	6r Lip(f)$ for any Lipshitz surjection $f : S_\Lambda\to T$, where
	$S_\Lambda$ is the universal cover of a wedge of circles marked by
	$\Lambda$. Therefore it suffices to produce an $F_r$-tree $S$ with
	quotient a wedge of $r$ circles and a map $f:S\to T$ so that
	$ Lip(f)\leq 2r-2$. The basis corresponding to the circles in the
	quotient of $S$ is then the desired basis.

	By Lemma~\ref{lem:efficient-tt} there is a simplicial tree $\Gamma''$
	and a map $f:\Gamma''\to T$ that is a composition of folds and
	collapses.  Thus the map $f:\Gamma''\to T$ has Lipshitz constant 1. The
	tree $\Gamma''$ is equivalent to one with no valence one or two
	vertices so $\Gamma''/F_r$ has at most $3r-3$ edges. By fixing a
	maximal tree $K\subseteq \Gamma''/ F_r$, the collapse of this maximal
	tree gives a wedge of circles $R$ with $r$ edges, and a homotopy
	equivalence $g:R\to \Gamma''/F_r$ with Lipshitz constant at most
	$\diam(K)\leq 2r-2$. The composition of the lift $\tilde{g}$ with $f$
	gives $f\circ\tilde{g}:\tilde{R}\to T$, which is the desired map.
\end{proof}
\section{Guiding examples}
\label{sec:examples}

When the Guirardel core of two Bass-Serre trees has no rectangles, its quotient
provides a simultaneous resolution of the two graphs of groups. This
construction immediately gives us a sufficient condition for two Dehn twists to
commute.

\begin{lemma}\label{lem:dt-com}
	Suppose $\tilde{\sigma}$ and $\tilde{\tau}$ are efficient Dehn twists based
	on graphs of groups $\bar{A}$ and $\bar{B}$ covered by $F_r$-trees $A$
	and $B$ respectively, representing $\sigma,\tau\in\Out(F_r)$. If
	$i(A,B) = 0$ then $[\sigma,\tau] = 1$ in $\Out(F_r)$.
\end{lemma}

\begin{proof}
	Since $A$ and $B$ are simplicial,
	$i(A,B) = 0$ implies that
	$\hatcore(A,B)$ is a tree. Therefore, by Lemma
	\ref{lem:core-bs},
	$\hatcore(A,B)$ is the Bass-Serre tree of a
	graph of groups $\Gamma$, and we may without loss of generality assume
	$\bar{A}$ and $\bar{B}$ fit into the
	following diagram, where $\pi_{\bar{A}}$ and $\pi_{\bar{B}}$ are quotient
	graph of groups morphisms that collapse edges.
	\[\xymatrix{
		& \Gamma\ar[ld]_{\pi_{\bar{A}}}\ar[rd]^{\pi_{\bar{B}}} & \\
			\bar{A} & & \bar{B}
	}\]
	Moreover (and this is still the content of Lemma \ref{lem:core-bs}),
	the edge groups of $\Gamma$ are edge groups of either 
	$\bar{A}$ or $\bar{B}$. 
	
	Define $\hat{\sigma}$ on $\Gamma$ by the system of twisters
	\[ z_e = \left\{ \begin{matrix} z_{\pi_{\bar{A}}(e)} &
		\pi_{\bar{A}}(e) \in E(\bar{A}) \\
		1 & \mbox{otherwise.}
	\end{matrix}\right.\]
	By construction, $\pi_{\bar{A}}\hat{\sigma} =
	\tilde{\sigma}\pi_{\bar{A}}$ at the level of the fundamental groupoid, so that
	$\hat{\sigma}$ is also a representative of $\sigma$. (The induced
	automorphism on the fundamental group coming from a graph of groups
	collapse is the identity~\cite{cohen-lustig}.) Similarly define
	$\hat{\tau}$, thus simultaneously realizing $\sigma$ and $\tau$ as Dehn
	twists on $\Gamma$, whence $[\sigma,\tau] = 1$.
\end{proof}

Towards a converse, Clay and Pettet give a partial result, using the notion of
a filling pair of Dehn twists~\cite{clay-pettet}. The other key tool is the
ping-pong lemma, which we use in the following formulation similar to the form
used by Clay and Pettet and Hamidi-Tehrani~\cites{clay-pettet,hamidi-tehrani}.

\begin{lemma}[Ping-Pong]
	Suppose $G = \langle a,b\rangle$ acts on a set $P$, and there is a
	partition $P = P_a \sqcup P_b$ into disjoint subsets such that $a^{\pm n}(P_b) \subseteq P_a$
	and $b^{\pm n}(P_a) \subseteq P_b$ for all $n > 0$. Then $G\cong F_2$.
\end{lemma}

\begin{proof}
	Any non-trivial reduced word is either a power of $a$ or
	conjugate to one of the form $w=a^{n_1}w'a^{n_2}$ for non-zero integers
	$n_1, n_2$ and $w'$ reduced starting and ending with a power of $b$.
	For $w$ in this form, $w(P_b)\cap P_b \subseteq P_a\cap P_b = \emptyset$, so
	$w\neq \id$.
\end{proof}

\begin{definition} 
	Let $X$ be a finitely generated group acting on $T$ a simplicial tree.
	The \emph{free $T$ volume of $X$}, $\vol_T(X)$ is the number
	of edges with trivial stabilizer in the graph of groups quotient of
	the minimal subtree $T^X\subset T$.
\end{definition}

Note that $\vol_T(\langle g\rangle) = \ell_T(g)$ for $g\in X$.

\begin{definition}
	Two graphs of groups $\bar{A}$ and $\bar{B}$ associated to $F_r$-trees
	$A$ and $B$ \emph{fill} if for every proper free factor or infinite
	cyclic subgroup $X\leq F_r$,
		\[ \vol_A(X) +\vol_B(X) > 0. \]
\end{definition}

\begin{definition}
	Suppose $\tilde{\sigma},\tilde{\tau}$ are representatives of Dehn
	twists based on $\bar{A}$ and $\bar{B}$, where both graphs of
	groups have one edge and fundamental group $F_r$. If $\bar{A}$ and
	$\bar{B}$ fill then we call the induced outer automorphisms
	$\sigma$ and $\tau$ a \emph{filling pair}.
\end{definition}

This definition is a close parallel to the notion of a pair of filling simple
closed curves, and Clay and Pettet strengthen this parallel to a theorem.

\begin{theorem}[\cite{clay-pettet}*{Theorem 5.3}]
	Suppose $\sigma, \tau\in \Out(F_r)$ are a filling pair of Dehn twists.
	Then there is an $N$ such that
	\begin{enumerate}
		\item $\langle \sigma^N,\tau^N\rangle\cong F_2$			\footnote{This conclusion holds under the weaker
			assumption that the two twists are
			\emph{hyperbolic-hyperbolic}
			(Definition~\ref{def:hyp-hyp}).}
		\item If $\phi\in\langle \sigma^N,\tau^N\rangle$ is not
			conjugate to a generator then $\phi$ is an atoroidal
			fully irreducible outer automorphism.
	\end{enumerate}
\end{theorem}

In developing their definition of free volume, Clay and Pettet use the
Guirardel core as motivation, but give a form suited explicitly to the proof of
their theorem. The definition of filling is indeed noticed by the core.

\begin{proposition}\label{prop:fill-core}
	Suppose $\bar{A}$ and $\bar{B}$ are small graphs of groups with one
	edge and fundamental group $F_r$ that fill. Then the Bass-Serre trees
	$A$ and $B$ have $i(A,B) > 0$ and the action of $F_r$ on $\core(A,B)$ is free.
\end{proposition}

\begin{proof}
	First, for any $(p,q)\in \core$, and $x\neq\id \in F_r$,  we have
	$\ell_A(x) = \vol_A(<x>)$ and $\ell_B(x) = \vol_B(<x>)$. Since
	$\bar{A}$ and $\bar{B}$ fill,
		\[ \ell_A(x)+\ell_B(x) = \vol_A(<x>) + \vol_B(<x>) > 0\]
	and therefore $(p,q)\cdot x \neq (p,q)$.

	To see that the core contains a rectangle we will show that the two
	trees have incompatible combinatorics~(Lemma \ref{lem:compat-tree}). To fix notation let $e$ be the
	edge of $\bar{A}$ and $f$ be the edge of $\bar{B}$. Let $\bar{A}_e =
	\langle c\rangle$. If $o(e)\neq t(e)$, let $a\in \bar{A}_{o(e)}$ be
	an element with no power conjugate into $\iota_{\bar{e}}(\bar{A}_e)$,
	and $\beta\in
	\bar{A}_{t(e)}$
	be an element not conjugate into $\iota_e(\bar{A}_e)$. Set 
	$b=e\beta e^{-1}$ in $\pi_1(\bar{A},o(e))$. If $o(e) = t(e)$ take $a$
	as before and $b = e$ in $\pi_1(\bar{A},o(e))$. 
	
	By construction,
	$\ell_A(ab) > 0$, and so
	$ab$ is not conjugate to $\iota_{\bar{e}}(c)$. Again, by the
	filling property, since $\ell_A(c) = 0$,
	$\ell_B(c) > 0$. Since $ab$ and $c$ are not
	conjugate, the characteristic sets of $ab$ and $c$ in $B$
	meet in at most a finite number of edges of $C_{c}^B$,
	since $B$ is small. Thus there is some $n > 0$ such that
	$C_{ab}^B\cap C_{c^{-n}ab c^n}^B = \emptyset$. However, by construction
	$C_{ab}^A$ contains the arc in
	$A$ stabilized by $c$, so
	$C_{ab}^A\cap C_{c^{-n}ab c^n}^A$
	contains this arc for all $n$. Therefore the two Bass-Serre trees are
	incompatible, the core contains a rectangle, and since both trees are
	simplicial this implies that the intersection number is positive, as
	required.
\end{proof}

This proposition motivates a variation of Clay and Pettet's result, in pursuit of a
converse to Lemma~\ref{lem:dt-com}. This variation cannot make the stronger
assertion that the generated group contains an atoroidal fully irreducible
element. Indeed, take $\sigma$ and $\tau$ to
be a filling pair of Dehn twists for $F_k$ and consider the automorphism
$\sigma\ast \id_m$ and $\tau\ast\id_m$ acting on $F_k\ast F_m$. This is a pair
of Dehn twists of $F_{k+m}$ that has powers generating a free group, but does
not fill, and every automorphism in $\langle \sigma\ast\id_m,\tau\ast\id_m\rangle$
fixes the conjugacy class of the complementary $F_m$ free factor, so all elements of the
generated group represent reducible outer automorphisms. Nevertheless, there is a partial converse to
Lemma~\ref{lem:dt-com}, finding free groups generated by 
pairs of Dehn twists based on one-edge graphs of groups using a variation on
their argument.

\begin{definition}\label{def:hyp-hyp}
	Suppose $\bar{A}$ and $\bar{B}$ are minimal visible small graphs of
	groups with one edge and associated $F_r$-trees $A$ and $B$. The pair
	is \emph{hyperbolic-hyperbolic} if both for the edge $e\in E(\bar{A})$,
	a generator $z_e$ of $\bar{A}_e$ acts hyperbolically on $B$; and for
	the edge $f\in E(\bar{B})$, a generator $z_f$ of $\bar{B}_f$ acts
	hyperbolically on $A$.
\end{definition}

\begin{proposition}
	Suppose $\bar{A}$ and $\bar{B}$ are minimal visible small graphs
	of groups with one edge. If $\bar{A}$ and $\bar{B}$ are hyperbolic-hyperbolic, then
	$i(A,B) > 0$.
\end{proposition}

\begin{proof}
	The proof of Proposition~\ref{prop:fill-core} applies immediately to
	show that the two Bass-Serre trees are not compatible. The construction
	used only the positive translation length of $\ell_B(c)$
	for a generator $c$ of an edge group of $\bar{A}_e$ and that $\bar{B}$ is
	small.
\end{proof}

\begin{remark}
	As noted in the proof, the above proposition is much more general,
	giving a sufficient condition for incompatibility: for any two minimal,
	visible, small graphs of groups, if there is an edge of one with a
	generator hyperbolic in the other then the core of the Bass-Serre trees
	has a rectangle.
\end{remark}

The hyperbolic-hyperbolic condition is sufficient to give a length function
ping-pong argument similar to Clay and Pettet's.

\begin{lemma}\label{lem:dt-free}
	Suppose $\tilde{\sigma}$ and $\tilde{\tau}$ are efficient Dehn twist
	representatives of $\sigma,\tau \in \Out(F_r)$, on one-edge graphs of
	groups $\bar{A}$ and $\bar{B}$ respectively. If $\bar{A}$
	and $\bar{B}$ are hyperbolic-hyperbolic, then for any $n \geq N = 48r^2-48r+3$ the
	group $\langle\sigma^n,\tau^n\rangle \cong F_2$.
\end{lemma}

\begin{proof}
	Let $e$ denote the edge of $\bar{A}$, $\bar{A}_e = \langle a\rangle$, $f$
	the edge of $\bar{B}$ and $\bar{B}_f = \langle b \rangle$. Let $s, t$ be
	nonzero integers so that the twisters of $\tilde{\sigma}$ and
	$\tilde{\tau}$ are $z_e = a^s$ and $z_f = b^t$ respectively. We will
	conduct a ping-pong argument similar to Clay and Pettet's free factor
	ping pong technique. Consider the partitioned subset of conjugacy classes
	$P = P_\sigma \sqcup P_\tau$ defined by,
	\begin{align*}
		P_\sigma &= \{ [w]\in P |  \ell_A(w) <
		\ell_B(w) \} \\
		P_\tau &= \{ [w]\in P | \ell_B(w) <
		\ell_A(w) \}. \\
	\end{align*}
	This is a non-trivial partition, $a\in P_\tau$ and
	$b\in P_\sigma$ by hypothesis.

	Our goal then is to find a power $N$ depending only on the rank 
	such that for all $n \geq N$, $\sigma^{\pm n} (P_\tau)
	\subseteq P_\sigma$ and $\tau^{\pm n}(P_\sigma) \subseteq P_\tau$. By 
	the ping-pong lemma, this implies $\langle
	\sigma^n,\tau^n\rangle \cong F_2$, as required. The argument will be
	symmetric.

	Suppose $[w]\in P_\tau$, so that $\ell_A(w) > 0$. Fix a
	cyclically reduced representative in transverse Bass-Serre normal form with
	respect to an ordered basis $\Lambda$ of $F_r$ based at a vertex of $\bar{A}$:
	\[ w = e_1a^{k_1}w_1e_2a^{k_2}w_2\cdots e_\ell a^{k_\ell}w_\ell \]
	where $\ell = \ell_A(w)$, $e_i\in\{e,\bar{e}\}$, we are suppressing the different edge morphisms sending $a$ into
	relevant vertex groups, and each $w_i$ is in the right transversal of
	the image of $a$ in the vertex group involved.
	Let $C$ be the bounded cancellation constant for the fixed basis of
	$F_r$ basis into $B$. With respect to this basis, after an
	appropriate conjugation we have the cyclically reduced conjugacy class
	representative $w'$ satisfying
		\[ |w'| = |a^{k_1'}|+\cdots+|w_{\ell-1}'|+|a^{k_\ell'}|+|w_\ell'| \]
	where $w_i'$ is the reduced word in this basis for the group element
	represented by the arrow $a^{\pm 1}w_ie_{i+1}a^{\pm 1}$ after
	collapsing a maximal tree, and $k_i'$ differs from $k_i$ by a fixed
	amount depending only on the edges $e_i$ and $e_{i+1}$, as each where each
	$w_i'$ might might disturb a fixed number of adjacent copies of
	conjugates of $a$ depending on the particular spelling (this follows
	from the minimality of the Schreier transversals used in transverse
	normal form). We have
	\[ \ell_A(w)  > \ell_B(w) \geq \left(\sum
	|k_i'|\right)\ell_B(a) - 2C(\Lambda,B)\ell_A(w). \]
	Re-writing, we conclude
	\begin{equation}\label{eqn:df-free}
		\sum |k_i'| <
		\left(\frac{1+2C(\Lambda,B)}{\ell_B(a)}\right)\ell_A(w).
		\tag{$\dag$}
	\end{equation}
	
	Using the Dehn twist representative of $\sigma$, we calculate 
		\[ \tilde{\sigma}^n(w) = e_1a^{\epsilon_1s
		n}a^{k_1}w_1e_2a^{\epsilon_2sn}a^{k_2}w_2\cdots
		e_\ell a^{\epsilon_\ell sn}a^{k_\ell}w_\ell \]
	where $\epsilon_i \in \{\pm 1\}$ according to the orientation of $e$
	represented by $e_i$. Reducing these words, and applying bounded
	cancellation in the same fashion we have
	\begin{align*}
		\ell_B(\tilde{\sigma}^n(w)) &\geq \sum_{i=1}^\ell
		(|\epsilon_i sn + k_i'|)\ell_B(a)- 2C(\Lambda, B)\ell_A(w) \\
		&\geq \left(|sn|\ell_A(w)-\sum
		|k_i'|\right)\ell_B(a)-2C(\Lambda, B)\ell_A(w) \\
		&\geq \left(
	|sn|-\frac{1+2C(\Lambda, B)}{\ell_B(a)}\right)\ell_A(w)\ell_B(a)-2C(\Lambda, B)\ell_A(w)
	\end{align*}
	with the last step following from Equation \ref{eqn:df-free}. Thus we
	have
	\[
		\frac{\ell_B(\tilde{\sigma}^n(w))}{\ell_A(\tilde{\sigma}^n(w))}
		=
		\frac{\ell_B(\tilde{\sigma}^n(w))}{\ell_A(w)}
		\geq \left(
		|sn|-\frac{1+2C(\Lambda,
	B)}{\ell_B(a)}\right)\ell_B(a)-2C(\Lambda, B).
		\]
	Therefore, to ensure $\sigma^n(w)\in P_\sigma$ we require
	\[\left(|sn|-\frac{1+2C(\Lambda, B)}{\ell_B(a)}\right)\ell_B(a)-2C(\Lambda, B)
		> 1 \]
	that is,
	\[ |n| > \frac{2+4C(\Lambda, B)}{|s|\ell_B(a)}. \]
	Since $|s|\ell_B(a) \geq 1$, having $|n| > 2+4C(\Lambda, B)$ suffices.
	By the Lemma~\ref{lem:good-basis} there is some basis $\Lambda$ such
	that $C(\Lambda,B)\leq 6r(2r-2)$, any choice of order on this basis
	will do. Let  $N = 48r^2-48r+3$.
	The preceding calculation implies that for all $|n|\geq N$,
	$\sigma^{\pm n}(P_\tau)\subseteq P_\sigma$. By a similar calculation,
	(using a good basis $\Lambda'$ so that $C(\Lambda', A) \leq 6r(2r-2)$), we find that for
	any $|n| \geq N$, $\tau^{\pm n}(P_\sigma)\subseteq P_\tau$.
	Therefore the group $\langle \sigma^N, \tau^N\rangle$ acting on $P = P_\sigma \sqcup P_\tau$ satisfies
	the hypotheses of the ping-pong lemma, and we conclude $\langle
	\sigma^N,\tau^N\rangle \cong F_2$ as required.
\end{proof}

\begin{remark}
	The reader familiar with Cohen and Lustig's skyscraper lemma and
	parabolic orbits theorem may wonder why these facts did not feature in
	the above proof. Both of these tools are not strong enough
	to give the uniform convergence necessary to carry out a ping-pong type
	argument on $\overline{CV}_r$; the skyscraper lemma has constants that depend on the
	particular skyscraper involved, and the parabolic orbits theorem gives
	pointwise convergence of length functions on conjugacy classes but does
	not control the rate of convergence. A priori, this rate could be very
	bad, as demonstrated by the examples of Bestvina, Feighn, and
	Handel~\cite{bfh-ii}*{Remark 4.24}.
\end{remark}

Together Lemmas \ref{lem:dt-com} and \ref{lem:dt-free} come very close to a
proof of Theorem \ref{thm:dt-out-mccarthy}.
Nature is not so kind, and there are incompatible graphs of groups that are not
hyperbolic-hyperbolic.

\begin{example}
	\label{ex:elip-elip}
	Let $A$ and $C$ be the Bass-Serre trees of the following graphs of groups
	decompositions of $F_3$.
	
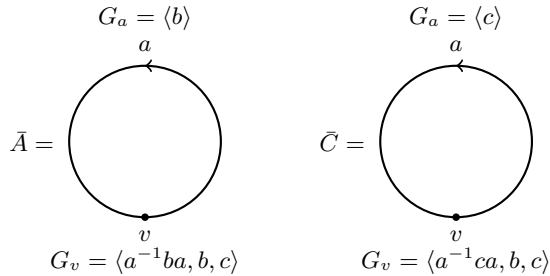
\begin{center}
		\begin{tikzpicture}[thick,font=\small,every text node
	part/.style={align=center}]
	\node [vertex,label=below:{$v$ \\ $G_{v} = \langle
	a^{-1}ba,b,c\rangle$}] (P) at (0,-1) {};
	\draw (0,0) circle (1cm)  [postaction={decorate},decoration={
    markings,
	mark=at position 0.25 with {\arrow{>}}}];
	\coordinate (T) at (0,1);
	\coordinate (A) at (-1,0);
	\node [left =2pt of A] {$\bar{A} =$};
	\node [above =2pt of T] {$G_a = \langle b\rangle$ \\ $a$};
\end{tikzpicture}
\qquad
\begin{tikzpicture}[thick,font=\small,every text node
	part/.style={align=center}]
	\node [vertex,label=below:{$v$ \\ $G_{v} = \langle
	a^{-1}ca,b,c\rangle$}] (P) at (0,-1) {};
	\draw (0,0) circle (1cm)  [postaction={decorate},decoration={
    markings,
	mark=at position 0.25 with {\arrow{>}}}];
	\coordinate (S) at (0,1);
	\coordinate (B) at (-1,0);
	\node [left =2pt of B] {$\bar{C} =$};
	\node [above =2pt of S] {$G_a = \langle c\rangle$ \\ $a$};
\end{tikzpicture}
 	\end{center}

	Let $\sigma$ and $\rho$ be the Nielsen transformations represented by
	Dehn twists about $\bar{A}$ and $\bar{C}$ by $b$ and $c$ respectively,
	so that 
	\begin{align*}
		\sigma(a) &= ba & \rho(a) &= ca \\
		\sigma(b) &= b & \rho(b) &= b \\
		\sigma(c) &= c & \rho(c) &= c.
	\end{align*}
	We claim that $\core(A,C)$ has a rectangle, so that $i(A,C) > 0$.
	Indeed, focus on the edges $e\subseteq A$ and $f\subseteq C$, each on
	the axis of $a$ with the induced orientation and the given edge
	stabilizers, illustrated below.
	
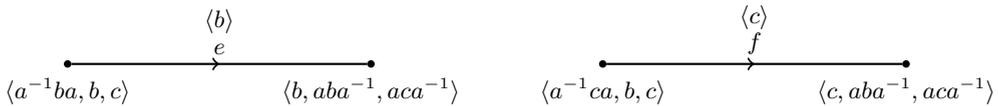
\begin{center}
		\begin{tikzpicture}[thick,font=\small,every text node
	part/.style={align=center}]
	\node [vertex,label=below:{$\langle a^{-1}ba,b,c\rangle$}] (a) at
	(-2,0) {};
	\node [vertex,label=below:{$\langle b,aba^{-1},aca^{-1}\rangle$}] (b)
	at (2,0) {};
	\draw (a) -- (b) [postaction={decorate},decoration={
    markings,
	mark=at position 0.5 with {\arrow{>}}}] node [midway,above] {$\langle b \rangle$
	\\ $e$ };
\end{tikzpicture}
\qquad
\begin{tikzpicture}[thick,font=\small,every text node
	part/.style={align=center}]
	\node [vertex,label=below:{$\langle a^{-1}ca,b,c\rangle$}] (a) at
	(-2,0) {};
	\node [vertex,label=below:{$\langle c,aba^{-1},aca^{-1}\rangle$}] (b)
	at (2,0) {};
	\draw (a) -- (b) [postaction={decorate},decoration={
    markings,
	mark=at position 0.5 with {\arrow{>}}}] node [midway,above] {$\langle c \rangle$
	\\ $f$ };
\end{tikzpicture}
 	\end{center}
	Note that $a\in\ebox{e}\cap\ebox{f}$ and
	$a^{-1}\in\ebox{\bar{e}}\cap\ebox{\bar{f}}$. Further, investigation of
	the diagrams shows that $b^{-1}ab\in\ebox{e}\cap\ebox{\bar{f}}$ and
	$c^{-1}ac\in \ebox{\bar{e}}\cap\ebox{f}$, so by
	Lemma~\ref{lem:compat-tree}, $e\times f\subseteq\core(A,B)$.

	This example is not hyperbolic-hyperbolic; $\ell_A(c) = \ell_C(b) = 0$.
	Nevertheless $\langle \sigma,\rho\rangle\cong F_2$. Indeed,
	$\omega\mapsto \omega(a)a^{-1}$ describes an isomorphism
	$\langle\sigma,\rho\rangle\cong\langle b,c\rangle$.
\end{example}

\begin{example}
	\label{ex:elip-hyp}
	Let $A$ and $B$ be the Bass-Serre trees of the following graphs of group
	decompositions of $F_3$.
	
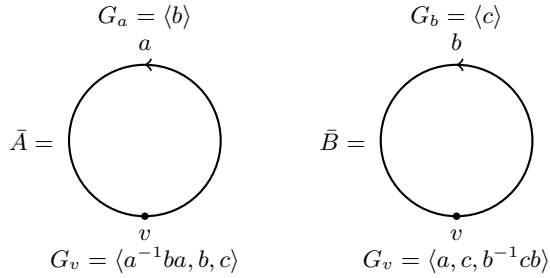
\begin{center}
		\begin{tikzpicture}[thick,font=\small,every text node
	part/.style={align=center}]
	\node [vertex,label=below:{$v$ \\ $G_{v} = \langle
	a^{-1}ba,b,c\rangle$}] (P) at (0,-1) {};
	\draw (0,0) circle (1cm)  [postaction={decorate},decoration={
    markings,
	mark=at position 0.25 with {\arrow{>}}}];
	\coordinate (T) at (0,1);
	\coordinate (A) at (-1,0);
	\node [left =2pt of A] {$\bar{A} =$};
	\node [above =2pt of T] {$G_a = \langle b\rangle$ \\ $a$};
\end{tikzpicture}
\qquad
\begin{tikzpicture}[thick,font=\small,every text node
	part/.style={align=center}]
	\node [vertex,label=below:{$v$ \\ $G_{v} = \langle
	a,c,b^{-1}cb\rangle$}] (P) at (0,-1) {};
	\draw (0,0) circle (1cm)  [postaction={decorate},decoration={
    markings,
	mark=at position 0.25 with {\arrow{>}}}];
	\coordinate (S) at (0,1);
	\coordinate (B) at (-1,0);
	\node [left =2pt of B] {$\bar{B} =$};
	\node [above =2pt of S] {$G_b = \langle c\rangle$ \\ $b$};
\end{tikzpicture}
 	\end{center}

	Let $\sigma$ and $\tau$ be the Nielsen transformations represented by
	Dehn twists about $\bar{A}$ and $\bar{B}$ by $b$ and $c$ respectively,
	so that 
	\begin{align*}
		\sigma(a) &= ba & \tau(a) &= a \\
		\sigma(b) &= b & \tau(b) &= cb \\
		\sigma(c) &= c & \tau(c) &= c.
	\end{align*}
	Again we have a rectangle in $\core(A,B)$. Consider $g = a,
	h=bab^{-1}$. Calculating with length functions we have
	\[ \ell_A(g) = \ell_A(h) = 1 \qquad \ell_B(g) = \ell_B(h) = 0\]
	and also
	\begin{gather*}
		\ell_A(gh) = 2 \neq 0 = \ell_A(gh^{-1}) \\
		\ell_B(gh) = \ell_B(gh^{-1}) = 2 > 0 = \ell_B(g)+\ell_B(h).
	\end{gather*}
	Therefore $A$ and $B$ do not have compatible combinatorics, so by
	Lemma~\ref{lem:compat-tree} $\core(A,B)$ has a rectangle and $i(A,B) >
	0$.

	This example is also not hyperbolic-hyperbolic, $\ell_B(b) = 1$ but
	$\ell_A(c) = 0$. Again, however, $\langle\sigma^3,\tau^3\rangle \cong F_2$.
	For a ping-pong set we use $P = \{ wa \in F_3 | w\in \langle
	b,c\rangle\}$ reduced words ending in $a$, and ping-pong partition $P_\sigma
	= \{ wb^{\pm 2}a \}$ and $P_\tau = P\setminus P_\sigma$. For all $N\neq
	0$, we have $\sigma^{3N}(P_\tau) \subseteq P_\sigma$ and
	$\tau^{3N}(P_\sigma) \subseteq P_\tau$. Note that it is only out of an
	aesthetic desire to use the same power of $N$ on both generators that
	we use $\tau^3$, it is the case that $\tau^N(P_\sigma) \subseteq
	P_\tau$ for all $N\neq 0$.
\end{example}

Both of these examples are presented with respect to a particularly nice basis,
and by taking the associated homotopy equivalence of the wedge of three circles
marked by the given basis, we see that all automorphisms in the above example
are upper triangular with respect to a fixed filtration. Both ping-pong arguments rely on
the interaction between the suffixes in this particular upper triangular
setting. This suggests a dichotomy, either length function ping-pong is
possible, or every element of the group generated by a pair of Dehn twists is polynomially
growing. To analyze the growth of elements in a subgroup of $\Out(F_r)$ generated by a pair of
Dehn twists we will follow the cue of Bestvina, Feighn, and Handel, and
understand the growth in topological models associated to the Dehn twists.

\section{Simultaneous graphs of spaces and normal forms}
\label{sec:sgos}

Guirardel gives a topological interpretation of the intersection number of two simplicial
$F_r$-trees.

\begin{theorem}[\cite{guirardel-core}*{Theorem 7.1}]\label{thm:top-core}
	Given two non-trivial
	simplicial $F_r$ trees $A$ and $B$ there exists a cell complex $X$ with
	$\pi_1(X) \cong F_r$ and two 2-sided subcomplexes $Y_A, Y_B \subset X$
	intersecting transversely such that $i(A,B) = |\pi_0(Y_A\cap Y_B)|$.
\end{theorem}

The space $X$ is constructed from the core. Let $\tilde{X} = \hatcore(A,B)\times
T$ where $T \cong \tilde{R}_r$ is the universal cover of a fixed wedge of $r$ circles.
Let $M_A$ be the set of midpoints of edges of $A$ and $M_B$ be the set of
midpoints of edges of $B$. The spaces $\tilde{Y}_A = \pi_A^{-1}(M_A)\times T$
and $\tilde{Y}_B = \pi_B^{-1}(M_B)\times T$ are a family of two-sided
subcomplexes of $\tilde{X}$. The connected components of $\tilde{Y}_A\cap
\tilde{Y}_B$ are of the form $x\times T$ where $x$ is a point in the interior
of a 2-cell of $\hatcore(A,B)$ or a midpoint of an edge in
$\hatcore\setminus\core$. The intersections of the form $x\times T$ when $x$ is
a midpoint of an edge in the augmented core are not transverse, indeed $x\times
T$ is a connected component of both $\tilde{Y}_A$ and $\tilde{Y}_B$ in this
case. A transverse intersection can be obtained by instead using $M_B'$ and
an equivariant choice of points in the interior of the edges of $B$ none of which
are the midpoints, denote this perturbation of $\tilde{Y}_B$ by $\tilde{Y}_B'$.  The
connected components of $\tilde{Y}_A \cap \tilde{Y}_B'$ are in one-to-one correspondence
with the 2-cells of $\hatcore(A,B)$. The quotients by the diagonal $F_r$
action, denoted $X, Y_A,$ and $Y_B'$ respectively, are the desired spaces.

These quotient spaces can be viewed through the lens of model spaces for graphs
of groups, discussed in Section~\ref{sec:bs-theory}. Let $\bar{A}$ and $\bar{B}$ be the graphs
of groups covered by $A$ and $B$ respectively. The compositions
$\pi_A\circ\pi_{\hatcore}$ and $\pi_B\circ\pi_{\hatcore}$ of projection maps
descend to the quotient and give maps $q_A:X\to\bar{A}$ and $q_B: X\to\bar{B}$.
These maps make $X$ a graph of spaces over $\bar{A}$ and $\bar{B}$ simultaneously, with
the connected components of $Y_A$ and $Y_B$ in the role of edge spaces. Denote
by $\mathcal{A}$ and $\mathcal{B}$ the graphs of spaces structures on $X$
induced by $q_A$ and $q_B$ respectively, with $\mathcal{A}_v = q^{-1}_A(v)$ the
vertex space over $v\in V(\bar{A})$, $\mathcal{A}_e = q^{-1}(e)$ the mapping
cylinder over the midpoint space $\mathcal{A}_e^m = q^{-1}(m_e)$ of an edge
$e\in E(\bar{A})$, and similar notation for $\mathcal{B}$. The goal of this
section is to establish a normal form for paths and circuits in a simultaneous
graph of spaces. This behavior of the core is captured in the
following definition.

\begin{definition}\label{def:sgos}
	Let $\bar{A}$ and $\bar{B}$ be two $F_r$ graphs of groups. A
	complex $X$ is a \emph{simultaneous graph of spaces resolving $\bar{A}$ and
	$\bar{B}$} if there are maps $q_A: X\to \bar{A}$ and $q_B:X\to\bar{B}$
	making $X$ a graph of spaces for $\bar{A}$ and $\bar{B}$ respectively
	(the induced structures denoted $\mathcal{A}$ and $\mathcal{B}$),
	and the following conditions on subspaces are satisfied:
	\begin{enumerate}
		\item The midpoint spaces $\mathcal{A}_e^m$ and
			$\mathcal{B}_f^m$ are either equal or intersect
			transversely for all edges $e\in E(\bar{A})$ and $f\in
			E(\bar{B})$.
		\item The intersection $\mathcal{A}_v\cap\mathcal{B}_e$ is the
			mapping cylinder for the maps of $\mathcal{A}_v\cap
			\mathcal{B}_e^m$ into
			$\mathcal{A}_v\cap\mathcal{B}_{o(e)}$ and
			$\mathcal{A}_v\cap\mathcal{B}_{t(e)}$ as a sub-mapping
			cylinder of $\mathcal{B}_e$.
	\end{enumerate}
	The \emph{core} of $X$ is the subcomplex
	\[ \bigcup_{\substack{e\in E(\bar{A})\\ f\in E(\bar{B})}}
	\mathcal{A}_e\cap\mathcal{B}_f \]
	A subcomplex $Y=\mathcal{A}_e\cap\mathcal{B}_f$ of the core is
	\emph{twice-light} if $\mathcal{A}_e^m = \mathcal{B}_f^m$.
\end{definition}

\begin{corollary}\label{cor:top-core}
	For any two $F_r$ graphs of groups $\bar{A}$ and $\bar{B}$ there is a
	simultaneous graph of spaces resolving them.
\end{corollary}

\begin{proof} 
	The space $X$ constructed in the proof of Theorem
	\ref{thm:top-core} from the core of the Bass-Serre trees covering
	$\bar{A}$ and $\bar{B}$ is the desired space.
\end{proof}

\begin{remark}
	When $X = \hatcore \times_{F_r} T$, the core of $X$ is the closure of
	the preimages of the interiors of the 2-cells of $\hatcore$ and the
	edges of $\hatcore\setminus\core$. The latter are the twice-light
	subcomplexes.
\end{remark}

Edges $e\subseteq X^{(1)}$ in the 1-skeleton of a simultaneous graph of
spaces fall into a taxonomy given by the two decompositions. Recall that in a
single graph of spaces structure $\mathcal{X}$, an edge in $X^{(1)}$ is
\emph{$\mathcal{X}$-nodal} if it lies in a vertex space, and
\emph{$\mathcal{X}$-crossing} otherwise. We extend this terminology to a
simultaneous graph of spaces.

\begin{definition}
	Let $e\subseteq X^{(1)}$ be an edge in the 1-skeleton of a
	simultaneous graph of spaces resolving $\bar{A}$ and $\bar{B}$. We say
	$e$ is
	\begin{description}
		\item[nodal] if it is both $\mathcal{A}$- and $\mathcal{B}$-nodal,
		\item[$\mathcal{A}$-crossing] if it is
			$\mathcal{A}$-crossing but $\mathcal{B}$-nodal,
		\item[$\mathcal{B}$-crossing] if it is $\mathcal{B}$-crossing
			but $\mathcal{A}$-nodal,
		\item[double-crossing] if it is both $\mathcal{A}$-crossing and
			$\mathcal{B}$-crossing.
	\end{description}
\end{definition}

The possible ambiguity of terminology will be avoided by always making clear
whether we are considering a single graph of spaces structure or a simultaneous
graph of spaces structure.

For a single graph of spaces, based paths have a normal form that gives a
topological counterpart to the Bass-Serre normal form for the fundamental
groupoid. Recall Lemma~\ref{lem:gos-nf}, that every path based in the one
skeleton of a graph of spaces is homotopic relative to the endpoints to a path
	\[ v_0H_1v_1H_2\cdots H_nv_n \]
where each $v_i$ is a (possibly trivial) tight edge path of
$\mathcal{X}$-nodal edges, each $H_i$ is $\mathcal{X}$-crossing, and for
all $1\leq i \leq n-1,$ $H_iv_iH_{i+1}$ is not homotopic relative to the
endpoints to an $\mathcal{X}$-nodal edge path. A similar normal form is
possible in a simultaneous graph of spaces.

\begin{lemma}
	\label{lem:sgos-nf}
	Every path in $X$, a simultaneous graph of spaces resolving $\bar{A}$
	and $\bar{B}$, is homotopic relative to the endpoints to a path of the
	form (called \emph{simultaneous normal form})
	\[ W_{0,0} K_{0,1}W_{0,1}\cdots K_{0,n_0}W_{0,n_0}H_1W_{1,0}\cdots H_m
	W_{m,0} K_{m,1}\cdots K_{m,n_m} W_{m,n_m} \]
	where the $W_{i,j}$ are (possibly trivial) tight edge paths of nodal
	edges, the $K_{i,j}$ are $\mathcal{B}$-crossing edges, and the $H_i$
	are either $\mathcal{A}$-crossing or double-crossing edges. Further this
	path is in normal form for both $\mathcal{A}$ and $\mathcal{B}$, so
	that the number of $\mathcal{B}$-crossing edges plus double-crossing edges
	and the number of $\mathcal{A}$-crossing edges plus double-crossing edges
	are both invariants of the relative homotopy class of the path. A
	similar statement holds for free homotopy classes of loops.
\end{lemma}

\begin{proof}
	Throughout this proof all homotopies will be homotopies of paths
	relative to the endpoints. Suppose $\gamma$ is a path in $X$. First, by
	Lemma \ref{lem:gos-nf}, $\gamma$ is homotopic to a path in
	$\mathcal{A}$-normal form
		\[v_0H_1v_1H_2v_2\cdots H_mv_m \]
	with each $v_i$ an $\mathcal{A}$-nodal path and each $H_i$ either
	$\mathcal{A}$-crossing or double-crossing. With respect to $\mathcal{B}$,
	each $v_i$ is an edge path, not necessarily in normal form, of the form
		\[ W_{i,0}K_{i,1}W_{i,1}\cdots K_{i,n_i} W_{i,n_i} \]
	where each $W_{i,j}$ is $\mathcal{B}$-nodal (and so nodal in the
	simultaneous graph of spaces) and each $K_{i,j}$ is
	$\mathcal{B}$-crossing (in the simultaneous graph of spaces sense).  
	We can take this path to $\mathcal{B}$-normal form by erasing pairs of
	crossing edges, but we must do so without introducing
	$\mathcal{A}$-crossing edges. 
	
	Suppose for some $i$ the path
	$K_{i,j}W_{i,j}K_{i,j+1}$ is homotopic to a path $W_{ij}'$ that is
	$\mathcal{B}$-nodal. Suppress the common index $i$. Let $p$ be the
	vertex of $\bar{A}$ such that $K_jW_jK_{j+1}\subseteq\mathcal{A}_p$,
	$e$ the edge of $\bar{B}$ such that $K_jW_jK_{j+1}\subseteq
	\mathcal{B}_e$, so that $W_j\subseteq \mathcal{B}_{t(e)}$ and
	$W_j'\subseteq \mathcal{B}_{o(e)}$. Since $K_jW_jK_{j+1}\subseteq
	\mathcal{A}_p\cap\mathcal{B}_e$, this is a path in the mapping cylinder
	for the inclusions of $\mathcal{A}_p\cap \mathcal{B}_e^m$ into the
	endpoints, and $W_j$ is a fiber of this cylinder. Thus $K_jW_jK_{j+1}$
	is homotopic via a homotopy in $\mathcal{A}_p\cap\mathcal{B}_e$ to a path $W_j'' \subseteq
	\mathcal{A}_p\cap\mathcal{B}_{o(e)}$. Using $W_j''$ to erase the pair
	of crossing edges, we see that each $v_i$ can be expressed in
	$\mathcal{B}$ normal form and remain $\mathcal{A}$-nodal. Thus
	$\gamma$ is homotopic to a path of the form
	\[ W_{0,0} K_{0,1}W_{0,1}\cdots K_{0,n_0}W_{0,n_0}H_1W_{1,0}\cdots H_m
	W_{m,0} K_{m,1}\cdots K_{m,n_m} W_{m,n_m}. \]
	
	This path may not be in $\mathcal{B}$-normal form. There are two
	possible cases, and in both we will show that it is possible to erase a
	pair of $\mathcal{B}$-crossing edges without destroying
	$\mathcal{A}$-normal form. 
	
	First, suppose this path is not $\mathcal{B}$-normal because there is
	some $i$ such that $K_{i,n_i}W_{i,n_i}H_{i+1}$ (or symmetrically
	$H_iW_{i,0} K_{i,1}$) is homotopic to a path $W_i'$ that is
	$\mathcal{B}$-nodal. Let $f$ be the edge of $\bar{B}$ crossed by
	$K_{i,n_i}$. In this case, the endpoints map to $o(f)$ by $q_B$, and
	$q_B(W_{i,0}) = t(f)$ so by continuity $q_B(H_{i+1}) = \bar{f}$; thus $H_{i+1}$ is double-crossing. Note
	that this path is already in $\mathcal{A}$-normal form.  Again suppress
	the common index, and take $K_n W_n H$ to a path in the
	$\mathcal{B}$-vertex space $\mathcal{B}_{o(f)}$. This path will have some
	number of $\mathcal{A}$-crossing edges, but similar to the previous
	paragraph, this path is homotopic to one in $\mathcal{A}$-normal form
	via a homotopy inside
	$\mathcal{B}_{o(f)}$, so that by Lemma \ref{lem:gos-nf} $K_nW_n H$ is
	homotopic to a path of the form $W_n'H'W'$ with exactly one
	$\mathcal{A}$-crossing edge, and $W_n'$ and $W'$ are nodal.

	Second, suppose the resulting path is not $\mathcal{B}$-normal because there is
	some $i$ such that $K_{i,n_i}W_{i,n_i}H_{i+1}W_{i+1,0}K_{i,0}$ is
	homotopic to a path $W_i'$ that is $\mathcal{B}$-nodal, contained in
	the vertex space of $q\in V(\bar{B})$. In this case
	$H_{i+1}$ must be $\mathcal{A}$-crossing. As before, the path
	$W_i'\subseteq \mathcal{B}_q$ is homotopic to a path in $\mathcal{A}$-normal form
	contained in $\mathcal{B}_q$.

	In both cases, the number of $\mathcal{A}$-crossing edges is
	maintained, so the result is in $\mathcal{A}$-normal form.

	Therefore, a path $\gamma$ is homotopic to a path in simultaneous
	normal form, and can be taken to this normal form by composing the
	following homotopies:
	\begin{enumerate}
		\item take $\gamma$ to $\mathcal{A}$-normal form,
		\item take each $\mathcal{A}$-nodal sub-path to
			$\mathcal{B}$-normal form within the appropriate
			$\mathcal{A}$ vertex space,
		\item erase remaining pairs of $\mathcal{B}$-crossing edges,
			maintaining $\mathcal{A}$-normal form.
	\end{enumerate}

	The homotopy invariance of the number of crossing edge types follows
	immediately from Lemma~\ref{lem:gos-nf}.
\end{proof}

\section{Twisting in graphs of spaces}
\label{sec:tgos}

A Dehn twist on a graph of groups can be realized by an action on based
homotopy classes of paths in a graph of spaces. Let $\Gamma$ be a
graph of groups modeled by the graph of spaces $X$, and $D$ a Dehn twist based
on $\Gamma$. Each crossing edge $H \in X^{(1)}$ lies over some edge
$e\in E(\Gamma)$. For each crossing edge $H$ pick a loop $\gamma_H$ in
$\mathcal{X}_{t(e)}$, contained in the image of $\mathcal{X}_e\times\{1\}$ representing $z_e$ and
based at $t(H)$. The action of $D$ on a crossing edge is the concatenation
	\[ D(H) = H\gamma_H\]
The action is extended to an action on all paths in $X^{(1)}$ by concatenation and $D(v) =
v$ for every nodal path, and to based homotopy classes by taking
one-skeleton representatives. That this action is well-defined and represents the
Dehn twist $D$ faithfully follows from noting that the below diagram of
fundamental groupoids commutes.
\[\xymatrix{
	\pi_1(X,X^{(0)}) \ar[r]^{D}\ar[d] & \pi_1(X,X^{(0)})\ar[d] \\
	\pi_1(\Gamma)\ar[r]_{D} & \pi_1(\Gamma)
	}\]
Also from this diagram we see that if a path $\gamma$ is in normal form, then
so is $D(\gamma)$, with the same crossing edges. 

Extending this to the setting of a simultaneous graph of spaces resolving
$\bar{A}$ and $\bar{B}$, and twists $\tilde{\sigma}$ based on $\bar{A}$ and
$\tilde{\tau}$ based on $\bar{B}$, we see that $\tilde{\sigma}$ preserves
$\mathcal{A}$-normal form (though we can make no comment on the
$\mathcal{B}$-normal form) and a symmetric statement holds for $\tilde{\tau}$.
To understand the behavior of paths in simultaneous normal form we must track
the extent to which $\tilde{\sigma}$ alters the number of $\mathcal{B}$-crossing
edges and vise-versa. This interaction is contained entirely in the graphs of
groups, and applies to all twists based on the graphs.

\begin{definition}\label{def:etg}
	The \emph{edge twist digraph} $\mathcal{ET}(\bar{A},\bar{B})$ of two
	small graphs of groups is a directed graph with vertex set
		\[ V(\mathcal{ET}) = \{(e,\bar{e}), | e\in
		E(\bar{A})\}\cup\{ (f,\bar{f}) | f\in E(\bar{B})\}, \]
	directed edges $((e,\bar{e}),(f,\bar{f}))$
	$e\in E(\bar{A}), f\in E(\bar{B})$ when a generator $\bar{A}_e =
	\langle z_e\rangle$ or its inverse uses $f$ or $\bar{f}$ in cyclically
	reduced normal form with respect to $\bar{B}$, and directed edges
	$((f,\bar{f}),(e,\bar{e}))$ $f\in E(\bar{B}), e\in E(\bar{A})$ when a
	generator $\bar{B}_f = \langle z_f\rangle$ or its inverse uses $e$ or
	$\bar{e}$ in cyclically reduced normal form with respect to $\bar{A}$.
\end{definition}

\begin{example}
	Suppose $\bar{A}$ and $\bar{B}$ are hyperbolic-hyperbolic one-edge
	graphs of groups with edges $a$ and $b$ respectively. Then
	$\mathcal{ET}(\bar{A},\bar{B})$ is the following digraph.
	
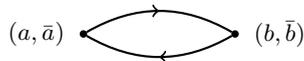
\begin{center}
		\begin{tikzpicture}[thick,font=\small,every text node
	part/.style={align=center}]
	\node [vertex] (A) at (-1,0) {};
	\node [vertex] (B) at (1,0) {};
	\draw [postaction={decorate},decoration={markings,mark=at
	position 0.5 with {\arrow{>}}}] (A) to[bend left] (B);
	\draw [bend left=30,postaction={decorate},decoration={markings,mark=at
	position 0.5 with {\arrow{>}}}] (B) to[bend left] (A);
	\node [left =2pt of A] {$(a,\bar{a})$};
	\node [right =2pt of B] {$(b,\bar{b})$};
\end{tikzpicture}
 	\end{center}
\end{example}

\begin{example}
	Let $\bar{A}$, $\bar{B}$, and $\bar{C}$ be the graphs of groups from
	Examples \ref{ex:elip-elip} and \ref{ex:elip-hyp}. Their edge twist
	digraphs are the following. Both $\mathcal{ET}(\bar{A},\bar{C})$ and
	$\mathcal{ET}(\bar{B},\bar{C})$ have two verticies and no edges. The
	digraph $\mathcal{ET}(\bar{A},\bar{B})$ has a single directed edge from
	$(a,\bar{a})$ to $(b,\bar{b})$.
\end{example}

This definition is made somewhat cumbersome by the presence of orientation. The
vertex set is the unoriented edges of the two graphs of groups, and the
property of crossing an unoriented edge in normal form is shared by the
generator and its inverse. We encapsulate the resulting awkwardness here, so
that subsequent arguments about paths in simultaneous normal form are clear.

The edge-twist structure controls the growth rate of elements in any group
generated by a twist $\tilde{\sigma}$ on $\bar{A}$ and $\tilde{\tau}$ on $\bar{B}$.

\begin{lemma}\label{lem:et-pg}
	Suppose $\bar{A}$ and $\bar{B}$ are minimal visible small graphs of
	groups with free fundamental group $F_r$ and $\mathcal{ET}(\bar{A},\bar{B})$
	is acyclic. Then for any pair of Dehn twists $\sigma, \tau\in
	\Out(F_r)$ represented by $\tilde{\sigma}$ based on $\bar{A}$ and
	$\tilde{\tau}$ based on $\bar{B}$, every element of $\langle
	\sigma,\tau\rangle \leq \Out(F_r)$ is polynomially growing. Moreover,
	the growth degree is at most the length of the longest directed path in
	$\mathcal{ET}(\bar{A},\bar{B})$.
\end{lemma}

\begin{proof}
	Let $X = \hatcore(A,B)\times_{F_r} T$ be the simultaneous graph of spaces
	constructed from the augmented core of the Bass-Serre trees $A$ and $B$
	for $\bar{A}$ and $\bar{B}$, with $\bar{T}$ a wedge of circles,
	equipped with the $\ell_1$ metric. Note that $\tilde{X}$ has an
	equivariant Lipschitz surjection to $T$ given by projection and that
	this descends to a Lipschitz homotopy equivalence on the quotient,
	denoted $\rho: X\to \bar{T}$. Further, if $\gamma$ is a loop in
	$X^{(1)}$ representing a conjugacy class $[g]$ of $\pi_1(\bar{T})$, 
		\[ \ell_T(g) \leq |\rho(\gamma)|_{\bar{T}} \leq Lip(\rho)\cdot
		|\gamma|_X\]
	where $|\cdot|$ is the arclength. Further, for any $w\in \langle
	\sigma,\tau\rangle$, by expressing $w$ as a word in the generators we
	get an action on paths $\tilde{w}$, with the property that $w(g)$ is
	represented by $\tilde{w}(\gamma)$.

	Therefore, it suffices to give a polynomial bound on the growth of
	paths in $X$ under the topological representatives of $\sigma$ and
	$\tau$. Moreover, for any edge path $\gamma$ the growth under the
	action of $\tilde{\sigma}$ and $\tilde{\tau}$ is bounded by the number
	of $\mathcal{A}$-crossing edges of $\gamma$ times the growth of
	$\mathcal{A}$-crossing edges plus the similar quantity for
	$\mathcal{B}$-crossing edges. So it suffices to bound the growth of
	crossing edges. (Note, this is an upper bound, we make no attempt to
	understand cancellation that might happen, as a result these bounds
	could be quite bad.)

	First, as a technical convenience, replace $\bar{A}$ and $\bar{B}$ by
	the isomorphic graphs of groups constructed from $A$ and $B$ using a
	fundamental domain in each that is the image under projection of a fundamental domain
	for $\hatcore(A,B)$, so that the edge groups of
	edges in each of $\bar{A}$ and $\bar{B}$ whose orbits are covered by diagonals of the core are not just conjugate, but
	equal on the nose. This does not change the outer automorphism class of
	the Dehn twists under consideration, nor does it change the edge twist
	graph. 

	Suppose $D$ is a double-crossing edge of $X^{(1)}$ lying over $e\in
	E(A)$ and $f\in E(B)$, so that the edge group $\bar{A}_e = \bar{B}_f =
	\langle z\rangle$ with common generator $z$. The cyclically reduced
	normal form of $z$ with respect to $\bar{A}$ based at
	$t(e)$ is $\iota_e(z)$; and with respect to $\bar{B}$ based at $t(f)$ is $\iota_f(z)$.
	Since both of these normal forms for $z$ contain no edges, the vertices
	$(e,\bar{e})$ and $(f,\bar{f})$ of $\mathcal{ET}(\bar{A},\bar{B})$ have
	no outgoing edges. Moreover, we can choose a loop representing a
	generator $z$ that is nodal and based at $t(D)$, and alter the
	topological representatives of $\tilde{\sigma}$ and $\tilde{\tau}$ so
	that $\tilde{\sigma}(D) = \gamma^{a}$ and $\tilde{\tau}(D) =
	\gamma^{b}$, concatenations of either $\gamma$ or its reverse,
	according to the expression of the twisters of $\tilde{\sigma}$ about
	$e$ and $\tilde{\tau}$ about $f$ in terms of the generator $z$. Thus,
		\[ \tilde{\sigma}^{s_n}\tilde{\tau}^{t_n}\cdots\tilde{\sigma}^s_1\tilde{\tau}^{t_1}(D)
			= D\gamma^{a\sum s_i + b\sum t_i} \]
	which has edge length at most linear in $\sum |s_i|+\sum |t_i|$.

	Suppose $H$ is an $\mathcal{A}$- or $\mathcal{B}$-crossing edge of
	$X^{(1)}$ lying over $(e,\bar{e}) \in V(\mathcal{ET})$. Let $d_e$ be
	the length of the longest directed path in
	$\mathcal{ET}(\bar{A},\bar{B})$ starting at $e$. We will use the
	notation $\poly_d(x)$ to stand for some polynomial of degree $d$ in
	$x$, as we are looking for an upper bound and making no attempt to
	estimate coefficients.

	\begin{claim*}
		For any crossing edge $H$, the length of
		$\tilde{\sigma}^{s_n}\tilde{\tau}^{s_n}\cdots\tilde{\sigma}^{s_1}\tilde{\tau}^{t_1}(H)$
		is at most $\poly_{d_e+1}(\sum |s_i|+\sum |t_i|)$.
	\end{claim*}

	\begin{proof}
		For double-crossing edges, the argument in the previous
		discussion establishes this claim. It remains to establish the
		claim for edges that are crossing but not double-crossing.
		The proof is by induction on $d_e$. As the argument is
		symmetric, we will suppose $H$ is $\mathcal{A}$-crossing, so
		that $e\in E(A)$. 

		\emph{Base Case: $d_e = 0$}. Let $\gamma_H$ be a loop representing a
		generator $z_e$ of $\bar{A}_e$ based at $t(H)$ and in
		simultaneous normal form. Since $(e,\bar{e})$ has no outgoing
		edges in $\mathcal{ET}$, the loop $\gamma_H$ is
		$\mathcal{B}$-nodal. Let $a$ be the power so that
		$z_e^a$ is the $e$ twister of $\tilde{\sigma}$. Use
		$\gamma_e^a$ in the topological representative of
		$\tilde{\sigma}$. Then for any $s$, $\tilde{\sigma}^s(H) =
		H\gamma_e^{as}$ is a $\mathcal{B}$-nodal path, and we have
			\[ \tilde{\sigma}^{s_n}\tilde{\tau}^{t_n}\cdots\tilde{\sigma}^{s_1}\tilde{\tau}^{t_1}(H)
				= H\gamma_H^{a\sum s_i} \]
		which has edge length at most linear in $\sum |s_i| +\sum
		|t_i|$, as required.

		\emph{Inductive Step: $d_e > 0$}. Since $d_e >0$, $(e,\bar{e})$ has neighbors
		$(f_1,\bar{f_1}),\ldots,(f_k,\bar{f_k})$. As before, use a
		simultaneous normal form representative $\gamma_H$ for a
		generator $z_e$ of $\bar{A}_e$ based at $t(H)$, so that
		$\sigma(H) = \gamma_H^a$. Since
		$\gamma_H$ has an $\mathcal{A}$-nodal representative by
		definition, we have in simultaneous normal form
			\[ \gamma_H = W_0K_1\cdots K_mW_m \]
		where $K_i$ lies over either $f_{k_i}$ or $\bar{f}_{k_i}$ by
		the definition of the edge twist graph. Further, for each
		$f_i$, the longest path in $\mathcal{ET}$ based at $f_i$,
		has length at most $d_e-1$. Calculating, we have
		\begin{align*}
			\tilde{\sigma}^{s_n}\tilde{\tau}^{t_n}\cdots\tilde{\sigma}^{s_1}\tilde{\tau}^{t_1}(H)
			&=
			H\gamma_H^{as_n}\tilde{\sigma}^{s_n}\tilde{\tau}^{t_n}(\gamma_H^{as_{n-1}})\\
			&\cdot\tilde{\sigma}^{s_n}\tilde{\tau}^{t_n}\tilde{\sigma}^{s_{n-1}}\tilde{\tau}^{t_{n-1}}(\gamma_H^{as_{n-2}})\\
			&\vdots \\
			&\cdot
			\tilde{\sigma}^{s_n}\tilde{\tau}^{t_n}\cdots\tilde{\tau}^{t_2}(\gamma_H^{as_1}).
		\end{align*}
		By the induction hypothesis, the length of each $K_{v_i}$ under
		a composition of powers of $\tilde{\sigma}$ and $\tilde{\tau}$
		is bounded by a polynomial of degree at most $d_e$. Hence the
		path
		$\tilde{\sigma}^{s_n}\tilde{\tau}^{t_n}\cdots
		\tilde{\sigma}^{s_2}\tilde{\tau}^{t_2}(\gamma^{as_1})$ has length at most
			\[ |as_1|\cdot\poly_{d_e}\left(\sum_{i\geq 2}
			|s_i|+|t_i|\right). \]
		Similarly, we bound the lengths of the other components and
		estimate
		\begin{align*}
			|\tilde{\sigma}^{s_n}\tilde{\tau}^{t_n}\cdots\tilde{\sigma}^{s_1}\tilde{\tau}^{t_1}(H)|
			&\leq
			|H\gamma_H^{as_n}\tilde{\sigma}^{s_n}\tilde{\tau}^{t_n}(\gamma_H^{as_{n-1}})|\\
			&+|\tilde{\sigma}^{s_n}\tilde{\tau}^{t_n}\tilde{\sigma}^{s_{n-1}}\tilde{\tau}^{t_{n-1}}(\gamma_H^{as_{n-2}})|\\
			&\vdots \\
			&+|
			\tilde{\sigma}^{s_n}\tilde{\tau}^{t_n}\cdots\tilde{\tau}^{t_2}(\gamma_H^{as_1})|
			\\
			& \leq \sum_{i=1}^n
			|as_i|\cdot\poly_{d_e}\left(\sum_{j>i}
			|s_j|+|t_j|\right) \\
		\end{align*}
		and this quantity is in turn at most $\poly_{d_e+1}\left(\sum
		|s_i|+|t_i|\right)$. This completes the claim.
	\end{proof}

	Finally, suppose $w =
	\sigma^{s_n}\tau^{s_n}\cdots\sigma^{s_1}\tau^{t_1} \in
	\langle\sigma,\tau\rangle$. For any $g\in F$, let $\gamma$ be a loop in
	simultaneous normal form representing the conjugacy class of $g$ in
	$X^{(1)}$. The length $\ell_T(w^N(g))$ is bounded by the length in $X$
	of $\tilde{w}^N(\gamma)$, which by the claim is at most
	\[ \poly_{d+1}(N\cdot\left(\sum|s_i|+|t_i|\right)) \]
	where $d$ is the length of the longest directed path in $\mathcal{ET}$.
	This is a polynomial of degree $d+1$ in $N$, which completes the lemma.
\end{proof}

An interesting question, which we do not pursue here is whether or not Lemma
\ref{lem:et-pg} is sharp. That is, if $\mathcal{ET}(\bar{A},\bar{B})$ contains
a cycle, is there some pair of twists $\sigma, \tau$ with representatives based
on $\bar{A}$ and $\bar{B}$ respectively so that the group generated contains an
outer automorphism with an exponentially growing stratum? In the setting of
one-edge splittings, Clay and Pettet's result is in this direction: two
one-edge graphs of groups that fill have a directed cycle of length two in
their edge-twist graphs; the group generated contains a fully irreducible
element, which is exponentially growing.

\section{Dehn twists on incompatible graphs generate free groups}
\label{sec:main-theorem}

We are now in a position to give a full converse to Lemma~\ref{lem:dt-com}. The
proof is by two cases, decided by the structure of the edge-twist graph. When
the edge-twist graph contains a cycle, this cycle enables a length function
ping-pong argument that is almost identical to the proof of
Lemma~\ref{lem:dt-free}. When the edge-twist graph is acyclic, the group
generated by the pair of twists is polynomially growing and we analyze its
structure using the Kolchin theorem for $\Out(F_r)$ of Bestvina, Feighn, and
Handel. As the two arguments are significantly different, we present them as
two lemmas.

\begin{lemma}\label{lem:dt-cycle-free}
	Suppose $\sigma$ and $\tau$ are Dehn twists of $F_r$ with efficient
	representatives $\tilde{\sigma}$ and $\tilde{\tau}$ on graphs of groups
	$\bar{A}$ and $\bar{B}$ such that $\mathcal{ET}(\bar{A},\bar{B})$
	contains a cycle. Then, for all $n\geq N = 48r^2-48r+3$, the group
	$\langle\sigma^N,\tau^N\rangle \cong F_2$.
\end{lemma}

\begin{proof}
	Let $(u_1,\bar{u_1}),\ldots,(u_c,\bar{u}_c)$ and
	$(v_1,\bar{v}_1),\ldots,(v_c,\bar{v}_c)$ be the vertices of a primitive
	cycle in $\mathcal{ET}$, with $u_i\in E(\bar{A})$ and $v_i\in
	E(\bar{B})$. (It is psychologically unfortunate to use $u$ and $v$ for
	edges, but this usage is only for this proof.) The index $c$ is the same for both sets as $\mathcal{ET}$
	is bipartite, and no vertex $(u_i,\bar{u}_i)$ or $(v_i,\bar{v}_i)$ is
	repeated. For each edge $u_i$ fix a generator $\langle a_{u_i}\rangle =
	\bar{A}_{u_i}$ and $s_{u_i}\neq 0$ so that the twister of $\tilde{\sigma}$
	about $u_i$ is $z_{u_i} = a_{u_i}^{s_{u_i}}$, and similarly fix $\langle
	b_{v_i}\rangle = \bar{B}_{v_i}$ and $t_i\neq 0$. (The $s_u$ and $t_v$ are
	nonzero as both $\tilde{\sigma}$ and $\tilde{\tau}$ twist on every edge
	of their respective graphs.) Let $\bar{A}'$ and $\bar{B}'$ be the
	quotient graphs of groups obtained by collapsing
	$E(\bar{A})\setminus\{u_i,\bar{u_i}\}$ and
	$E(\bar{B})\setminus\{v_i,\bar{v}_i\}$ in $\bar{A}$ and $\bar{B}$
	respectively.

	We will again use conjugacy class ping-pong. Define a set $P$ by the
	partition $P= P_\sigma \sqcup P_\tau$ where
	\begin{align*}
		P_\sigma &= \{ [w]\in P | \ell_{A'}(w) < \ell_{B'}(w)
		\}\\
		P_\tau &= \{ [w]\in P | \ell_{B'}(w) < \ell_{A'}(w) \}.
	\end{align*}
	This partition is
	non-trivial, the edge-group generators $a_u \in P_\tau$ and $b_v\in P_\sigma$.

	Once more we will find an $N$ so that for all $n\geq N$, $\sigma^{\pm
	n}(P_\tau) \subseteq P_\sigma$ and $\tau^{\pm n}(P_\sigma) \subseteq
	P_\tau$, to conclude, by the ping-pong lemma, $\langle
	\sigma^n,\tau^n\rangle \cong F_2$. The argument will be symmetric, and
	almost identical to that of Lemma \ref{lem:dt-free}.

	Suppose $[w]\in P_\tau$, so that $0 < \ell_{A'}(w)$. Fix a cyclically
	reduced representative of $w$ in transverse Bass-Serre normal form with respect to
	a fixed basis $\Lambda$ and $\bar{A}'$,
		\[ w = e_1a_{e_1}^{k_1}w_1e_2a_{e_2}^{k_2}w_2\cdots
		e_\ell a_{e_\ell}^{k_\ell}w_\ell \]
	where we are suppressing the different edge morphisms, using $\ell =
	\ell_{A'}(w)$ for legibility,
	$e_i\in\{u_i,\bar{u}_i\}$, and each $w_i$ is in the right transversal
	of the image of $a_{e_i}$ in the vertex group involved.  Let $C$ be the
	bounded cancellation constant for the fixed basis of $F_r$ basis into
	$B'$. With respect to this basis, after an appropriate
	conjugation we have the cyclically reduced conjugacy class
	representative $w'$ satisfying
		\[ |w'| =
		|a_{e_1}^{k_1'}|+\cdots+|w_{\ell-1}'|+|a_{e_\ell}^{k_\ell'}|+|w_\ell'| \]
	where $w_i'$ is the reduced word for the group element
	obtained from the arrow $a_{e_i}^{\pm 1}w_ie_{i+1}a_{e_{i+1}}^{\pm 1}$
	after collapsing a maximal tree, and $k_i'$ differs from $k_i$ by a
	fixed amount depending only on the edges, as each
	$w_i'$ might might disturb a fixed number of adjacent copies of
	conjugates of $a_i$ depending on the particular spelling (This
	decomposition of $w'$ as a reduced word with respect to $\Lambda$
	follows from normal form as in the proof of Lemma \ref{lem:dt-free}.)
	Let $\alpha = \min_i \{\ell_{B'}(a_{u_i})\}$. Since each $(u_i,\bar{u}_i)$ is
	joined to some $(v_i,\bar{v}_i)$ by an edge in $\mathcal{ET}$ as they
	are all vertices of a cycle, $\alpha > 0$. We have, by bounded
	cancellation,
	\begin{align*} 
		\ell_{A'}(w) > \ell_{B'} &\geq \sum_{i=1}^\ell
		|k_i|\ell_{B'}(a_{e_i}) - 2C(\Lambda,B')\ell_{A'}(w) \\
		&\geq \left(\sum_{i=1}^p |k_i'|\right)\alpha - 2C(\Lambda,B')\ell_{A'}(w).
	\end{align*}
	We conclude
	\begin{equation}\label{eqn:dt-freer}
		\sum |k_i'| <
		\left(\frac{1+2C(\Lambda,B')}{\alpha}\right)\ell_{A'}(w).
		\tag{$\dag$}
	\end{equation}

	Calculating with the induced action of $\tilde{\sigma}$ on arrows in
	$\pi_1(\bar{A}')$ and
	abusing notation to also call this action $\tilde{\sigma}$, we have
		\[ \tilde{\sigma}^n(w) =
		e_1a_{e_1}^{s_{e_1}n}a_{e_1}^{k_1}\tilde{\sigma}^n(w_1)e_2a_{e_2}^{s_{e_2}n}a_{e_2}^{k_2}\cdots
		e_\ell
	a_{e_\ell}^{s_{e_\ell}n}a_{e_\ell}^{k_\ell}\tilde{\sigma}^n(w_\ell). \]
	The possibility that
	$\tilde{\sigma}^n(w_i)$ is of the form $a_{e_i}^{\epsilon
	n}x_ia_{e_{i+1}}^{\delta n}$ is ruled out by the no positive bonding condition of
	the efficient representative: $\epsilon n$ and $s_{e_i}n$ must have
	the same sign, and also $\delta n$ and $s_{e_{i+1}}n$. So, reducing and
	applying bounded cancellation in the same fashion, we have, with $s =
	\min_i \{ |s_i|\}$
	\begin{align*}
		\ell_{B'}(\sigma^n(w)) &\geq
		\sum_{i=1}^\ell|s_{e_i}n+k_i'|\ell_{B'}(a_{e_i}) -
		2C(\Lambda,B')\ell_{A'}(w)
		\\
		&\geq \left(|sn|\ell_{A'}(w)-\sum_{i=1}^p |k_i'|\right)\alpha -
		2C(\Lambda,B')\ell_{A'}(w)\\
		&\geq
		\left(|sn|-\frac{1+2C(\Lambda,B')}{\alpha}\right)\alpha\ell_{A'}(w)-2C(\Lambda,B')\ell_{A'}(w)
	\end{align*}
	with the last step following from Equation \ref{eqn:dt-freer}. Thus we have
	\[\frac{\ell_{B'}(\sigma^n(w))}{\ell_{A'}(\sigma^n(w))} =
		\frac{\ell_{B'}(\sigma^n(w))}{\ell_{A'}(w)} \geq
	\left(|sn|-\frac{1+2C(\Lambda,B')}{\alpha}\right)\alpha-2C(\Lambda,B').\]
	Therefore, to ensure $\sigma^n(w)\in P_\sigma$ we require
	\[ \left(|sn|-\frac{1+2C(\Lambda,B')}{\alpha}\right)\alpha -
		2C(\Lambda,B') >
		1 \]
	that is,
	\[ |n| > \frac{2+4C(\Lambda,B')}{s\alpha}. \]

	As before, after choosing bases using Lemma~\ref{lem:good-basis} for
	both this calculation and a similar calculation involving $\tau$, we
	conclude that for all $n\geq N= 48r^2-48r+3$ both $\sigma^{\pm
	n}(P_\tau)\subseteq P_\sigma$ and $\tau^{\pm n}(P_\sigma)\subseteq
	P_\tau$. Therefore the group $\langle
	\sigma^n,\tau^n\rangle \cong F_2$ by the ping-pong lemma.
\end{proof}

The presence of a cycle in $\mathcal{ET}(\bar{A},\bar{B})$ is essential in the
above proof; it guarantees there is some subset of twisters and edges where the
growth of one restricted length function is linear in the value of the other
restricted length function. Without a cycle, this kind of uniform control is
unavailable, as illustrated by Examples~\ref{ex:elip-elip} and
\ref{ex:elip-hyp}. Fortunately, this is the exact case where the generated group
is polynomially growing and the Kolchin theorem can be applied. Using the
simultaneous upper triangular representatives a different form of ping-pong can
be effected.

First we require a lemma relating the core of two efficient twists and the
structure of their simultaneous upper triangular representatives.
The contrapositive of this lemma will be used to find paths suitable for
ping-pong, after applying the Kolchin theorem.

\begin{lemma}\label{lem:upg-compat}
	Suppose $\sigma$ and $\tau$ are Dehn twist outer automorphisms with
	upper-triangular relative train-track representatives $\hat{\sigma}$
	and $\hat{\tau}$ with respect to a filtered graph $\emptyset =
	\Gamma_0\subsetneq\Gamma_1\subsetneq\cdots\subsetneq \Gamma_k =
	\Gamma$, and efficient representatives $\tilde{\sigma}$ and
	$\tilde{\tau}$ on graphs of groups $\bar{A}$ and $\bar{B}$ covered by
	$A$ and $B$ respectively. If 
	\begin{enumerate}
		\item Every suffix of $\hat{\sigma}$ is $\hat{\tau}$-Nielsen,
		\item Every suffix of $\hat{\tau}$ is $\hat{\sigma}$-Nielsen,
		\item For every edge $E_i \in \Gamma_i\setminus\Gamma_{i-1}$ if
			$E_i$ is a linear edge of both $\hat{\sigma}$ and $\hat{\tau}$ the associated primitive Nielsen
	paths are equal (up to orientation), 
	\end{enumerate}
	then $i(A,B) = 0$.
\end{lemma}

\begin{proof}
	The construction of efficient representatives in
	Lemma~\ref{lem:efficient-tt} from a relative train-track involves first
	folding conjugates, then
	a series of folding edges in linear families, and finally a series of
	graph of groups Stallings folds; it follows from Cohen and Lustig's
	parabolic orbits theorem that the simplicial structure of the resulting
	tree is unique (Theorem~\ref{thm:parabolic-orbits} and Corollary~\ref{cor:et-uniq}). We carry out the same construction, using both
	$\hat{\sigma}$-linear edges and $\hat{\tau}$-linear edges. A
	\emph{joint linear family} is a collection of single edges $\{E_i\}$
	which have either $\hat{\sigma}$ or $\hat{\tau}$ suffixes that are a power of a
	fixed primitive Nielsen path $\gamma$. By hypothesis, if two edges
	$E_i$ and $E_j$ are in the same linear family for one of the maps, then
	they are in the same joint linear family. As in the construction of
	efficient representatives, we first fold conjugates and then linear
	families; the hypotheses ensure that this can be done in a compatible
	fashion. The resulting folded graph and folded representatives, $\hat{\sigma}',\hat{\tau}':\Gamma'\to\Gamma'$ are
	still upper triangular, represent $\sigma$
	and $\tau$ respectively, and have the property that every linear family
	contains one edge.

	We now construct a tree $C$ that resolves the trees $A$ and $B$. First,
	recall that the efficient representative of $\hat{\sigma}'$ on a tree
	$A$ can be constructed from $\Gamma'$ as follows. Start
	with $A_0$ obtained from the universal cover of $\Gamma'$ by collapsing
	all $\hat{\sigma}'$ fixed edges of $\Gamma'$. We then work up the
	remaining orbits of edges of $A_0$ by the filtration of $\Gamma'$. If
	$\hat{\sigma}'(E_i) = E_i$ then set $A_i = A_{i-1}$, otherwise
	$\hat{\sigma}'(E_i) = E_iu_i$, and each lift of $u_i$ by construction represents an
	element in the vertex group based at a lift of $t(E_i)$; the tree $A_i$
	is obtained from $A_{i-1}$ by folding the associated
	primitive Nielsen path $\gamma_i$ over $E_i$ (the details are
	in Lemma~\ref{lem:efficient-tt}), and the result $A_k$ is $A$. To
	construct the resolving tree, we start with $C_0$, obtained from the
	universal cover of $\Gamma'$ by collapsing all edges that are fixed by
	both $\hat{\sigma}'$ and $\hat{\tau}'$. Then, working up the hierarchy
	of $\Gamma$, if $E_i$ is both $\hat{\sigma}'$ and $\hat{\tau}'$ fixed,
	set $C_i = C_{i-1}$, otherwise $\hat{\sigma}'(E_i) = E_i\gamma_i^s$ and
	$\hat{\tau}'(E_i) = E_i\gamma_i^t$ for a primitive Nielsen path
	$\gamma_i$ (allowing the possibility $s$ or $t$ is zero); in this case
	by construction lifts of $\gamma_i$ represent elements in the vertex
	stabilizers of lifts of $t(E_i)$, so we obtain $C_i$ from $C_{i-1}$ by
	pulling $\gamma_i$ over $E_i$. The desired resolving tree is $C = C_k$. It is readily apparent from this
	construction that $C$ maps to $A$ and $B$ by collapse maps: collapse
	any remaining $\sigma$ fixed edges of $C$ to obtain $A$, and any
	remaining $\tau$ fixed edges of $C$ to obtain $B$. 

	By Theorem~\ref{thm:core-refine}, since $A$ and $B$ have a common
	refinement, the core is a tree and therefore contains no rectangles,
	whence $i(A,B) = 0$.
\end{proof}

With the relationship between the core and upper triangular representatives
understood, we complete the remaining parts of the proof of the main theorem.

\begin{lemma}\label{lem:dt-pg-free}
	Suppose $\tilde{\sigma},\tilde{\tau}$ are efficient Dehn twists with
	trivial image in $GL(r, \mathbb{Z}/3\mathbb{Z})$ based on $\bar{A}$ and
	$\bar{B}$ respectively. If $\mathcal{ET}(\bar{A},\bar{B})$ is acyclic
	and $i(A,B) > 0$, then $\langle \sigma^3,\tau^3\rangle \cong F_2$.
\end{lemma}

For the proof we require some notation. For two paths
$\gamma,\delta\subseteq \Gamma$ with the same initial point, the \emph{overlap
length} is defined by $\theta(\gamma,\delta) =
\frac{1}{2}(\len_\Gamma([\gamma])+\len_\Gamma([\delta]) -
\len([\bar{\gamma}\delta]))$, where we use the metric on $\Gamma$ induced by
assigning each edge length one. We will often understand the overlap length by
calculating the \emph{common initial segment} of two tight paths, this is the
connected component of the intersection of lifts of $\gamma$ and $\delta$ based
at a common point. The length of this
segment is equal to the overlap length.

\begin{proof}
	By Lemma~\ref{lem:et-pg}, the group $\langle \sigma,\tau\rangle$ is a
	polynomially growing subgroup of $\Out(F_r)$. By Bestvina, Feighn, and Handel's criterion for
	unipotence~\cite{bfh-ii}*{Proposition 3.5}, the group $\langle
	\sigma,\tau\rangle$ is a unipotent polynomially growing subgroup.
	Therefore, by the Kolchin theorem
	for $\Out(F_r)$ (see Theorem~\ref{thm:kolchin}) there is a filtered graph
	$\emptyset = \Gamma_0\subsetneq \Gamma_1\subsetneq \cdots\Gamma_k =
	\Gamma$ with each step in the filtration a single edge, so that
	$\langle \sigma,\tau\rangle$ is realized as a group of upper-triangular
	homotopy equivalences of $\Gamma$ with respect to the filtration.
	Let $\hat{\sigma}$ and $\hat{\tau}$ be the realizations of the
	generators. Since $\sigma$ and $\tau$ are UPG, every $\hat{\sigma}$-periodic
	Nielsen path is $\hat{\sigma}$-Nielsen and every $\hat{\tau}$-periodic
	Nielsen path is $\hat{\tau}$-Nielsen.

	Since $i(A,B) > 0$, $\mathcal{C}(A,B)$ contains a rectangle, the
	contrapositive of Lemma~\ref{lem:upg-compat} implies that either (up to
	relabeling) $\hat{\sigma}$ has a linear edge $E_i$ with suffix $u_i$
	that grows linearly under $\hat{\tau}$ (as in Example~\ref{ex:elip-hyp}
	where, using the upper triangular representatives on the rose with
	edges $\bar{a}, \bar{b}, \bar{c}$, $\hat{\sigma}(\bar{a}) =
	\bar{a}\bar{b}$, and $\bar{b}$ is $\hat{\tau}$-linear); or there is an
	edge $E_i$ so that the $\hat{\sigma}$ and $\hat{\tau}$ suffixes are
	powers of primitive Nielsen paths which generate non-equal cyclic
	subgroups and both suffixes are Nielsen for both automorphisms (as in
	Example~\ref{ex:elip-elip} where, again using the representatives on
	the rose with edges $\bar{a}, \bar{b}, \bar{c}$, the $\hat{\sigma}$ and
	$\hat{\rho}$ suffixes of
	$\bar{a}$ are respectively $\bar{b}$ and $\bar{c}$). This gives two cases. In each
	case the proof generalizes the analysis of the appropriate guiding
	example.

	\emph{Case 1.} Let $E_i$ be the lowest edge in the filtration such that
	its suffix under one automorphism grows linearly under the other, 
	and without loss of generality suppose that the $\hat{\sigma}$ suffix
	$u_i$
	grows linearly under $\hat{\tau}$. We will use as a ping-pong set
	\[ P = \{ [\omega(E_i)] | \omega\in
	\langle\hat{\sigma},\hat{\tau}\rangle \} \]
	the orbits of (the based homotopy class of) $E_i$ under tightening after applying elements of the
	group generated by $\hat{\sigma}$ and $\hat{\tau}$. Since a tight path
	is a unique representative of a based homotopy class the proof will
	focus on the tight representatives and the homotopy class will be suppressed.
	All of these classes have tight representatives
	of the form $E_iw$ with $w\subseteq \Gamma_{i-1}$ a tight path based at
	$t(E_i)$, since the group
	is upper triangular with respect to this filtration. Let
	\[ P_\sigma = \{ p\in P | \theta([E_iu_i^3],p) \geq
	\theta([E_iu_i^3],[E_iu_i^2])\mbox{ or } \theta([E_i\bar{u}_i^3],p) \geq
	\theta([E_i\bar{u}_i^3],[E_i\bar{u}_i^2]) \} \]
	and $P_\tau = P\setminus P_\sigma$ be a partition of $P$. It is clear
	that $P$ and $P_\sigma$ are non-empty, and we will show in the course
	of the proof that $P_\tau$ is non-empty. Let
	$\gamma_k$ be the common initial segment of $[u_i^k]$ and $[u_i^{k+1}]$, and
	$\gamma_{-k}$ the common initial segment of $[\bar{u}_i^k]$ and
	$[\bar{u}_i^{k+1}]$. Note that $[\bar{u}_i^k\gamma_1] = \gamma_{-(k-1)}$,
	the paths $\gamma_j$ are an increasing sequence of paths, and that
	$[\hat{\sigma}(\gamma_j)] = \gamma_ju'$ 
	where $u'$ is the $\hat{\sigma}$-Nielsen path associated to an
	exceptional $\hat{\sigma}-$Nielsen subpath of the primitive
	$\hat{\sigma}-$Nielsen path associated to $u_i$ if one exists.

	We claim $\hat{\sigma}^{\pm 3N}(P_\tau) \subseteq P_\sigma$ for
	$N\neq 0$. The argument will be symmetric for negative powers, so
	suppose $N > 0$. Consider $E_iw\in P_\tau$; we calculate
	\[ [\hat{\sigma}^{3N}(E_iw)] = E_i[[u_i^{3N}]\hat{\sigma}^{3N}(w)]. \]
	We must show that $E_i\gamma_2$
	is the initial part of the path $[\hat{\sigma}^{3N}(E_iw)]$. To
	establish this it suffices to show that $[\hat{\sigma}^{3N}(w)]$ does
	not start with $\gamma_{-k}$ for some $k > 3N-1$. For a contradiction
	suppose $[\hat{\sigma}^{3N}(w)] = \gamma_{-k}w'$ for some $k> 3N-1$.
	Consider the $\hat{\sigma}$-canonical decomposition of $\gamma_{-k}w'$.
	Either this agrees with the $\hat{\sigma}$-canonical decomposition of
	$\gamma_{-k}$, or the last edge of $\gamma_{-k}$ participates in a
	maximal exceptional subpath of $w'$, so that the decomposition of
	$\gamma_{-k}w'$ is obtained from $\gamma_{-(k-1)}$ and some $w''$. In
	either case, since every edge of $w$ is lower than the linear family
	associated to $u_i$, $[\hat{\sigma}^{-3N}(w'')]$ does not overlap
	$[u_i^k]$ in $\gamma_k$, and we have
	\[ w = [\hat{\sigma}^{-3N}(\gamma_{-(k-1)}w'')] =
	\gamma_{-(k-1)}[\hat{\sigma}^{-3N}(w'')].\]
	Since $k > 3N-1$, this implies $E_iw \in P_\sigma$, but we supposed
	$E_iw\notin P_\sigma$. Therefore, $E_i[u_i^{3N}\hat\sigma^{3N}(w)]$ has
	$E_i\gamma_2$ as an initial segment, so that
	$\hat{\sigma}^{3N}(E_iw)\in P_\sigma$. The argument for negative
	powers is symmetric.

	Next we claim $\hat{\tau}^{\pm N}(P_\sigma)\subseteq P_\tau$ for
	$N\neq 0$. Let $v_i$ be the
	$\hat{\tau}$ suffix of $E_i$ (possibly trivial). Since $u_i$ grows
	linearly under $\tau$, $\gamma_1$ must contain a $\hat{\tau}$-linear
	edge or $\hat{\tau}$-linear
	exceptional path in its $\hat{\tau}$ decomposition. Neither $v_i$ nor
	$\bar{v}_i$, which are $\hat{\tau}$-Nielsen, can contain a
	$\hat{\tau}$-linear component in their $\hat{\tau}$-canonical
	decomposition as
	$v_i$ is a $\hat{\tau}$ suffix. A similar statement holds for
	$\gamma_{-1}$. Thus $v_i$ and $\bar{v_i}$ do not have $\gamma_2$ or
	$\gamma_{-2}$ as an initial segment. Consider the highest $\hat{\tau}$-linear
	edge of $\gamma_1$; since $\hat{\tau}$ is upper-triangular 
	this edge cannot be canceled when tightening $\hat{\tau}^{N}(\gamma_1)$, 
	so $[\hat{\tau}^{\pm N}(\gamma_2)]$ has at most
	$\gamma_1$ in common with $\gamma_2$ (and similarly at most
	$\gamma_{-1}$ in common with $\gamma_{-2}$). Finally, suppose
	$E_i\gamma_2w \in P_\sigma$ is a tight representative. By the minimality in the choice of $E_i$,
	the highest $\hat{\tau}$-linear edge of $w$ is at most the same height as
	that in $\gamma_2$, so the highest $\hat{\tau}$-linear edges of
	$\gamma_2w$ do not cancel in the tightening of
	$\hat{\tau}^{\pm N}(\gamma_2)\hat{\tau}^{\pm N}(w)$. Putting this all
	together, the result $[\hat{\tau}^{\pm N}(E_i\gamma_2 w)]$ has at most
	$E_i\gamma_1$ in common with $E_i[u_i^3]$. Applying similar reasoning to
	$E_i\gamma_{-2}w'$, we conclude $\hat{\tau}^{\pm
	N}(P_\sigma)\subseteq P_\tau$ (this shows in particular that $P_\tau$
	is non-empty). So by the ping-pong lemma $\langle
	\sigma^3,\tau^3\rangle \leq \langle \sigma^3,\tau\rangle \cong
	F_2$ as required.

	\emph{Case 2.} Suppose no $\hat{\sigma}$-suffix is
	$\hat{\tau}$-growing and vise-versa, and
	that there is an edge $E$ such that $\hat{\sigma}(E) = Eu$ and
	$\hat{\tau}(E) = Ev$, and the associated primitive Nielsen paths $u'$
	and $v'$ do not generate isomorphic subgroups of $\pi_1(\Gamma,t(E))$.
	Since $v$ is not $\hat{\sigma}$-growing it is $\hat{\sigma}$-periodic,
	thus $[\hat{\sigma}(v)] = v$; similarly $[\hat{\tau}(u)] = u$. By
	hypothesis, $u_\ast,v_\ast \in \pi_1(\Gamma,t(E))$ generate a rank two
	free group $G$. Further, for $\omega\in\langle
	\hat{\sigma},\hat{\tau}\rangle$, $\omega(E) = Ew$ for some path $w$ so
	that $w_\ast \in \langle u_\ast,v_\ast\rangle$. It is immediate that
	$\omega \mapsto w_\ast$ is an isomorphism, hence $\langle\sigma,\tau\rangle
	\cong F_2$.

	In either case $\langle \sigma^3,\tau^3\rangle \cong F_2$ as required.
\end{proof}

\begin{corollary}
	Suppose $\tilde{\sigma},\tilde{\tau}$ are efficient Dehn twists
	satisfying every hypothesis of Lemma~\ref{lem:dt-pg-free} except the
	condition on their image in $GL(r,\mathbb{Z}/3\mathbb{Z})$. Then
	$\langle \sigma^9, \tau^9 \rangle \cong F_2$.
\end{corollary}

\begin{proof}
	Since $\sigma$ and $\tau$ are Dehn twists, they are
	unipotent~\cite{cohen-lustig-ii}, and so $\sigma^3,\tau^3$ have trivial
	image in $GL(r,\mathbb{Z}/3\mathbb{Z})$.  Therefore, by the lemma
	$\langle (\sigma^3)^3, (\tau^3)^3\rangle \cong F_2$.
\end{proof}

The culmination of this effort is a proof of a uniform McCarthy-type theorem for
$\Out(F_n)$ in the linearly growing case.

\begin{theorem*}[\ref{thm:dt-out-mccarthy}]
	Suppose $\sigma$ and $\tau$ are linearly growing outer automorphisms of
	$F_r$.
	For $N=(48r^2-48r+3)|GL(r,\mathbb{Z}/3\mathbb{Z})|$ the subgroup
	$\langle \sigma^N,\tau^N\rangle$ is
	either abelian or free of rank two. Moreover, the latter case holds
	exactly when $i(A,B)>0$ for the Bass-Serre trees $A$ and $B$ of
	efficient representatives of Dehn-twist powers of $\sigma$ and $\tau$.
\end{theorem*}

\begin{proof}
	First, using train tracks Cohen and Lustig show that a unipotent
	linearly growing automorphism is a Dehn twist~\cite{cohen-lustig-ii}.
	Bestvina, Feighn, and Handel~\cite{bfh-ii}*{Proposition 3.5} show that
	every polynomially growing outer automorphism with trivial image in
	$GL(r,\mathbb{Z}/3\mathbb{Z})$ is unipotent. Let $U =
	|GL(r,\mathbb{Z}/3\mathbb{Z})|$, so that $\sigma^U$ and $\tau^U$ are
	Dehn twists with trivial image in $GL(r, \mathbb{Z}/3\mathbb{Z})$ and
	efficient representatives on graphs of groups $\bar{A}$ and $\bar{B}$.
	If $i(A,B) > 0$ then by either Lemma \ref{lem:dt-cycle-free} or
	\ref{lem:dt-pg-free}, since $N = 48r^2-48r+3$ is divisible by $3$, the
	group $\langle \sigma^{UN},\tau^{UN} \rangle \cong F_2$; otherwise by
	Lemma \ref{lem:dt-com}, $\sigma^{U},\tau^{U}$ commute.
\end{proof}

\begin{acknowledgements}
I thank Marc Culler for his guidance in the completion of my
thesis. I am also grateful to my committee, Daniel Groves, Lee Mosher, Peter
Shalen, and Kevin Whyte, for their careful reading of this work and helpful
remarks. Samuel Taylor and David Futer's encouragement was most welcome in
achieving the uniform power. The anonymous referee's thorough remarks improved
the exposition, I very much appreciate their enthusiasm for examples. Finally,
throughout writing my thesis I had great conversations with members of the
geometric group theory community too numerous to list here.
\end{acknowledgements}

\affiliationone{	Edgar A. Bering IV \\
	Department of Mathematics, Wachman Hall, Temple University, 1805 Broad
Street, Philadelphia, PA 19122 \\
U. S. A.
\email{edgar.bering@temple.edu}}

\end{document}